\documentclass[12pt]{amsart}
\usepackage[english]{babel}
\usepackage{amsmath}
\usepackage{amsthm}
\usepackage{anysize}
\usepackage{mathtools}
\usepackage{mathrsfs}
\usepackage{amssymb}
\usepackage{xcolor}
\usepackage{xfrac}
\usepackage{esint}
\usepackage{graphicx}
\usepackage{float}
\usepackage{array}
\usepackage{enumitem}
\usepackage[new]{old-arrows}
\usepackage[all,cmtip]{xy}
\usepackage{tikz-cd}
\usepackage{centernot}
\usepackage{stmaryrd}
\usepackage{comment}
\excludecomment{confidential}

\numberwithin{equation}{section}

\newtheorem{theorem}{Theorem}[section]

\newtheorem{lemma}[theorem]{Lemma}
\newtheorem{proposition}[theorem]{Proposition}
\newtheorem{definition}[theorem]{Definition}

\DeclareMathOperator{\Alg}{Alg}

\DeclareMathOperator{\ord}{ord}

\DeclareMathOperator{\alg}{alg}

\newcommand{\field}[1] {\mathbb{#1}}
\newcommand{\N}{\field{N}}

\newcommand{\R}{\field{R}}

\newcommand{\K}{\field{K}}

\def\a{\alpha}
\def\b{\beta}
\def\e{\varepsilon}
\def\D{\Delta}
\def\d{\delta}
\def\g{\gamma}
\def\G{\Gamma}
\def\l{\lambda}

\def\o{\omega}
\def\O{\Omega}
\def\p{\partial}
\def\r{\rho}

\def\s{\sigma}

\def\v{\varphi}

\def\ua{\uparrow}

\newcommand{\mc}{\mathcal}
\newcommand{\mf}{\mathfrak}

\begin{document}
	\title{Generating loops and isolas in semilinear elliptic BVP's}
\author{Juli\'an L\'opez-G\'omez, Juan Carlos Sampedro} \thanks{The authors have been supported by the Research Grants PGC2018-097104-B-I00 and PID2021-123343NB-I00 of the Spanish Ministry of Science and  Technology, and by the Institute of Interdisciplinar Mathematics of Complutense University. J. C. Sampedro
has been also supported by the PhD Grant PRE2019\_1\_0220 of the Basque Country Government.}
\address{Institute of Interdisciplinary Mathematics \\
	Department of Mathematical Analysis and Applied Mathematics \\
	Complutense University of Madrid \\
	28040-Madrid \\
	Spain.}
\email{julian@mat.ucm.es, juancsam@ucm.es}


\begin{abstract}
In this paper, we ascertain the global $\l$-structure of the set of positive and negative solutions
bifurcating from $u=0$ for the semilinear elliptic BVP
\begin{equation*}
\left\{\begin{array}{ll}
-d\Delta u= \lambda\langle \mathfrak{a},\nabla u\rangle+u+\lambda u^{2}-u^{q} & \text{ in } \Omega, \\
u=0 & \text{ on } \partial\Omega,
\end{array}\right.
\end{equation*}
according to the values of $d>0$ and the integer number $q\geq 4$. Figures \ref{F1}--\ref{F3} summarize
the main findings of this paper according to the values of $d$ and $q$. Note that the role played by the parameter $\l$ in this model is very special, because,  besides measuring the strength of the convection, it quantifies the amplitude of the nonlinear term $\l u^2$. We regard to this problem as a mathematical toy
to generate solution loops and isolas in Reaction Diffusion equations.
\end{abstract}

\maketitle

\section{Introduction}

\noindent This paper studies the positive and the negative solutions of the elliptic semilinear problem
\begin{equation}
\label{i.1}
\left\{\begin{array}{ll}
-d\Delta u= \lambda\langle \mathfrak{a},\nabla u\rangle+u+\lambda u^{2}-u^{q} & \text{ in } \Omega, \\
u=0 & \text{ on } \partial\Omega,
\end{array}\right.
\end{equation}
where $\Omega$ is a bounded domain of class $\mathcal{C}^2$ of the Euclidean space $\mathbb{R}^{N}$, $N\geq 1$, with boundary $\p\O$, $d> 0$ is the diffusion coefficient, $q\geq 4$ is an integer number, $\mathfrak{a}\in \mathbb{R}^{N}\backslash\{0\}$, $\mathfrak{a}=(a_{1},\cdots,a_{N})$, and $\lambda\in\mathbb{R}$. In \eqref{i.1} we are denoting  by $\langle \cdot,\cdot\rangle$ the Euclidean product of $\R^N$, i.e.,
$$
\langle x, y\rangle =\sum_{i=1}^N x_iy_i \quad \hbox{for every}\;\;
x=(x_1,...,x_N),\, y=(y_1,...,y_N)\in\R^N,
$$
and we regard $\l$ and $d$ as bifurcation parameters: $\l$ the primary one, and $d$ the secondary. The main goal of this paper is ascertaining the evolution of the global $\l$-bifurcation diagrams of positive and negative solutions of \eqref{i.1} as $d$ varies in $(0,+\infty)$.
\par
The problem \eqref{i.1} is a multidimensional counterpart of the $1$-dimensional prototype problem introduced by the authors in \cite[Sect. 7]{JJ4} in the very special case when $d=1$, $q=4$, $\mf{a}=1$ and $\Omega=(0,\pi)$. In this paper we are interested in analyzing how vary the admissible multidimensional bifurcation $\l$-diagrams as the diffusion coefficient $d$ varies in $(0,+\infty)$ according to the value of $q$. Although the value of $q\geq 4$ is irrelevant when dealing with positive solutions, its oddity is extremely significant when dealing with negative solutions. The assumption $q\geq 4$ is required to keep unchanged the structure of the set of positive and negative solutions in a neighborhood of $(\l,u)=(0,0)$. As in the simplest one-dimensional model, our main technical devices here invoke the local and global bifurcation techniques for Fredholm operators discussed by the authors in \cite{JJ4}.
\par
Throughout this paper, for any given $V\in\mc{C}(\bar\O)$, we denote by $\s_1[-\D+V]$ the principal eigenvalue of $-\D+V$ in $\O$ under Dirichlet boundary conditions. To simplify notations, we will set
$\s_1\equiv \s_1[-\D]$. Also, we denote by $\v_0$ any principal eigenfunction associated to $\s_1$.
\par
Our results depend on the size of the secondary parameter $d>0$ and on the concrete value of $q\geq 4$. To describe our main findings, we need to divide them into three different blocks.
\par
Suppose $d\in (0,\s_1^{-1})$. Then, the set of positive solutions bifurcating from $u=0$ consists of a single compact connected component, $\mathscr{C}^+$,  linking $(-\l_1(d),0)$ to $(\l_1(d),0)$, where
\begin{equation}
\label{i.2}
\lambda_{1}(d):=\frac{2}{|\mathfrak{a}|}\sqrt{d(1-d\s_{1})}>0,
\end{equation}
while the set of negative solutions bifurcating from $u=0$ consists of another compact connected component, $\mathscr{C}^-$,  linking $(-\l_1(d),0)$ to $(\l_1(d),0)$, if $q\geq 5$ is odd, as illustrated in the first
plot of Figure \ref{F1}. In this figure, and in all subsequent ones, we are representing the value of the parameter $\l$ in abscisas versus the norm $\|u\|_{W^{2,p}(\O)}$, for some $p>N$, if $u>0$, or versus $-\|u\|_{W^{2,p}(\O)}$ if $u<0$. As usual, $W^{2,p}(\O)$ stands for the Sobolev space of the functions $u\in L^{p}(\O)$ having distributional derivatives $D^{\alpha}u\in L^{p}(\O)$ for $|\alpha|\leq 2$, and we denote by $W_0^{2,p}(\O)$ the kernel of the trace operator $\mathscr{T}: W^{1,p}(\O)\to L^p(\p\O)$.

\begin{center}
	\begin{figure}[h!]

	\tikzset{every picture/.style={line width=0.75pt}} 
	
	\begin{tikzpicture}[x=0.75pt,y=0.75pt,yscale=-1,xscale=1]
	
	\draw    (159.46,44) -- (159.46,217) ;
	\draw [line width=2.25]    (258.46,130.5) -- (60.46,130.5) ;
	\draw    (460.46,44) -- (460.46,217) ;
	\draw [line width=2.25]    (559.46,130.5) -- (361.46,130.5) ;
	\draw [color={rgb, 255:red, 74; green, 144; blue, 226 }  ,draw opacity=1 ][line width=2.25]    (101,130) .. controls (128,109) and (155.08,46.99) .. (193,70) ;
	\draw [color={rgb, 255:red, 74; green, 144; blue, 226 }  ,draw opacity=1 ][line width=2.25]    (193,70) .. controls (256.08,109.99) and (234,104) .. (218.08,130.99) ;
	\draw [color={rgb, 255:red, 184; green, 233; blue, 134 }  ,draw opacity=1 ][line width=2.25]    (101,130) .. controls (41,182) and (165,223) .. (218.08,130.99) ;
	\draw [color={rgb, 255:red, 74; green, 144; blue, 226 }  ,draw opacity=1 ][line width=2.25]    (410,130) .. controls (437,109) and (464.08,46.99) .. (502,70) ;
	\draw [color={rgb, 255:red, 74; green, 144; blue, 226 }  ,draw opacity=1 ][line width=2.25]    (502,70) .. controls (565.08,109.99) and (543,104) .. (527.08,130.99) ;
	\draw [color={rgb, 255:red, 184; green, 233; blue, 134 }  ,draw opacity=1 ][line width=2.25]    (344,169) .. controls (352,164) and (361,167) .. (410,130) ;
	\draw [color={rgb, 255:red, 184; green, 233; blue, 134 }  ,draw opacity=1 ][line width=2.25]    (527.08,130.99) .. controls (484,179) and (589,203) .. (600,196) ;
	
	\draw (216,142) node [anchor=north west][inner sep=0.75pt]    {$\lambda _{1}(d)$};
	\draw (32,139) node [anchor=north west][inner sep=0.75pt]    {$-\lambda _{1}(d)$};
	\draw (530.08,140.99) node [anchor=north west][inner sep=0.75pt]    {$\lambda _{1}(d)$};
	\draw (400,140) node [anchor=north west][inner sep=0.75pt]    {$-\lambda _{1}(d)$};
	\draw (114,228) node [anchor=north west][inner sep=0.75pt]    {\quad $q\geq 5$ \hbox{is odd}};
	\draw (419,229) node [anchor=north west][inner sep=0.75pt]    {\quad $q\geq 4$ \hbox{is even}};
	\draw (166,30) node [anchor=north west][inner sep=0.75pt]    {$\| \cdot \| _{W{^{2,p}}}$};
	\draw (466,26) node [anchor=north west][inner sep=0.75pt]    {$\| \cdot \| _{W{^{2,p}}}$};
	\draw (225,63) node [anchor=north west][inner sep=0.75pt]    {$\mathscr{C}^{+}$};
	\draw (529,60) node [anchor=north west][inner sep=0.75pt]    {$\mathscr{C}^{+}$};
	\draw (552,199) node [anchor=north west][inner sep=0.75pt]    {$\mathscr{C}_{+}^{-}$};
	\draw (186,181) node [anchor=north west][inner sep=0.75pt]    {$\mathscr{C}^{-}$};
	\draw (346,172) node [anchor=north west][inner sep=0.75pt]    {$\mathscr{C}_{-}^{-}$};

	\end{tikzpicture}
		\caption{Two bifurcation diagrams when $d<\s_{1}^{-1}$}
		\label{F1}
	\end{figure}
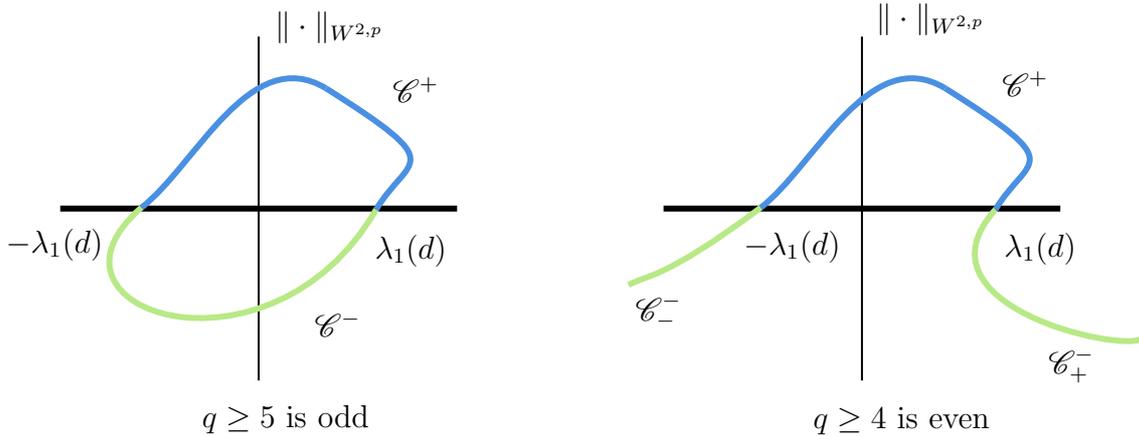
\end{center}

When $q\geq 4$ is an even integer and $N=1, 2$, or $q=4$ and $N=3$, we can prove that the global
bifurcation diagram of the negative solutions looks like shows the second plot of
Figure \ref{F1}, i.e., it contains two connected components of negative solutions, $\mathscr{C}^{-}_\pm$,
bifurcating from $(\pm\l_1(d),0)$, respectively,  such that
\begin{equation*}
(-\infty,-\lambda_{1}(d))\subset\mathcal{P}_{\l}(\mathscr{C}^{-}_{-}),\qquad
(\lambda_{1}(d),+\infty)\subset\mathcal{P}_{\l}(\mathscr{C}^{-}_{+}),
\end{equation*}
where $\mathcal{P}_{\l}$ stands for the $\l$-projection operator $\mathcal{P}_{\l}(\l,u):=\l$
for all $(\l,u)\in \R \times W^{2,p}(\O)$. We must impose $q=4$ when $N=3$ in order to get
\begin{equation}
\label{i.3}
  q<\frac{N+2}{N-2}
\end{equation}
and benefit of the existence of a priori bounds in $W^{2,p}_0(\O)$ for the negative solutions of
\eqref{i.1}, because these solutions are given through the positive solutions of a certain superlinear problem at $u=+\infty$.
\par
Further, as $d$ increases up to reach the critical value $\l=\s_1^{-1}$, we have that
\begin{equation}
\label{i.4}
  \lim_{d\ua \s_1^{-1}}\l_1(d)=0.
\end{equation}
Thus, the previous two bifurcation points from $u=0$, $(\pm\l_1(d),0)$, shrink to the single point $(0,0)$, which is the unique bifurcation point from $u=0$ to positive, or negative, solutions at $d=\s_1^{-1}$. Actually, assuming that $d=\s_1^{-1}$, we can prove that the set of positive solutions of \eqref{i.1} bifurcating from $(0,0)$ consists of a loop, $\mathscr{C}^+$, regardless whether $q\geq 4$ is even, or odd,  as illustrated by Figure \ref{F2}.

\begin{center}
	\begin{figure}[h!]

		\tikzset{every picture/.style={line width=0.75pt}} 
		
		\begin{tikzpicture}[x=0.75pt,y=0.75pt,yscale=-1,xscale=1]
		
		\draw    (179.46,64) -- (179.46,237) ;
		\draw [line width=2.25]    (278.46,150.5) -- (80.46,150.5) ;
		\draw    (480.46,64) -- (480.46,237) ;
		\draw [line width=2.25]    (579.46,150.5) -- (381.46,150.5) ;
		\draw [color={rgb, 255:red, 74; green, 144; blue, 226 }  ,draw opacity=1 ][line width=2.25]    (179.46,150.5) .. controls (179,99) and (203,72) .. (240,103) ;
		\draw [color={rgb, 255:red, 74; green, 144; blue, 226 }  ,draw opacity=1 ][line width=2.25]    (240,103) .. controls (251,115) and (239,129) .. (179.46,150.5) ;
		\draw [color={rgb, 255:red, 184; green, 233; blue, 134 }  ,draw opacity=1 ][line width=2.25]    (480.46,150.5) .. controls (482,212) and (521,218) .. (557,204) ;
		\draw [color={rgb, 255:red, 184; green, 233; blue, 134 }  ,draw opacity=1 ][line width=2.25]    (382,199) .. controls (390,194) and (430,169) .. (480.46,150.5) ;
		\draw [color={rgb, 255:red, 74; green, 144; blue, 226 }  ,draw opacity=1 ][line width=2.25]    (480.46,150.5) .. controls (480,99) and (504,72) .. (541,103) ;
		\draw [color={rgb, 255:red, 74; green, 144; blue, 226 }  ,draw opacity=1 ][line width=2.25]    (541,103) .. controls (552,115) and (540,129) .. (480.46,150.5) ;
		\draw [color={rgb, 255:red, 184; green, 233; blue, 134 }  ,draw opacity=1 ][line width=2.25]    (108,188) .. controls (118,173) and (150,161) .. (180.46,150.5) ;
		\draw [color={rgb, 255:red, 184; green, 233; blue, 134 }  ,draw opacity=1 ][line width=2.25]    (108,188) .. controls (90,215) and (178,221) .. (179.46,150.5) ;
		
		\draw (134,248) node [anchor=north west][inner sep=0.75pt]    {\quad $q\geq 5$ \hbox{is odd}};
		\draw (439,249) node [anchor=north west][inner sep=0.75pt]    {\quad $q\geq 4$ \hbox{is even}};
		\draw (186,50) node [anchor=north west][inner sep=0.75pt]    {$\| \cdot \| _{W{^{2,p}}}$};
		\draw (486,46) node [anchor=north west][inner sep=0.75pt]    {$\| \cdot \| _{W{^{2,p}}}$};
		\draw (245,83) node [anchor=north west][inner sep=0.75pt]    {$\mathscr{C}^{+}$};
		\draw (549,80) node [anchor=north west][inner sep=0.75pt]    {$\mathscr{C}^{+}$};
		\draw (564,214) node [anchor=north west][inner sep=0.75pt]    {$\mathscr{C}_{+}^{-}$};
		\draw (87,204) node [anchor=north west][inner sep=0.75pt]    {$\mathscr{C}^{-}$};
		\draw (391,206) node [anchor=north west][inner sep=0.75pt]    {$\mathscr{C}_{-}^{-}$};

		\end{tikzpicture}
		\caption{Two bifurcation diagrams when $d=\s_{1}^{-1}$}
		\label{F2}
	\end{figure}
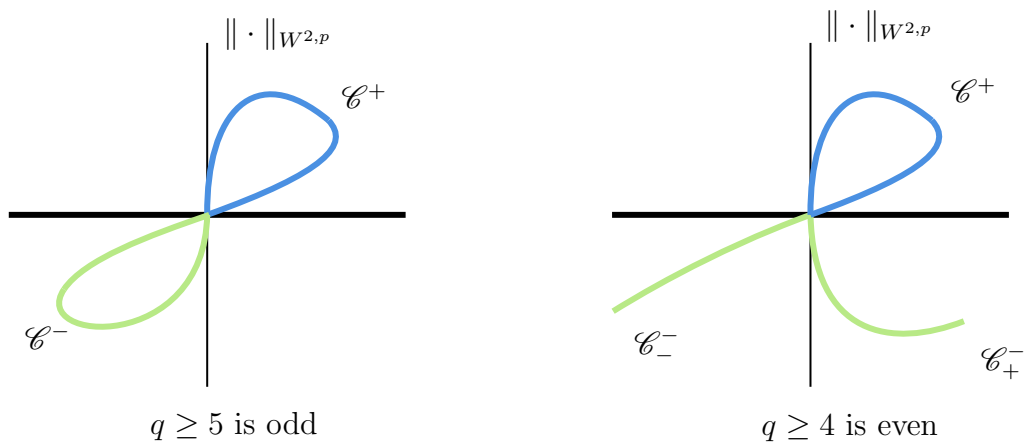
\end{center}

As in the previous case when $d<\s_1^{-1}$, the negative solutions of \eqref{i.1} can behave in a rather different manner, according to the values of $q$ and the spatial dimension $N$. For instance, as soon as $q\geq 5$ is an odd integer, the set of negative solutions of \eqref{i.1} bifurcating from $(0,0)$ consists of another loop, $\mathscr{C}^-$,
bifurcating from $(0,0)$, as sketched in the first plot of Figure \ref{F2}. However, when $q\geq 4$ is an even integer and $N=1, 2$, or $N=3$ and $q=4$, then the set of negative solutions emanating from $(0,0)$
consists of two disjoint connected components,  $\mathscr{C}^{-}_{-}$ and $\mathscr{C}^{-}_{+}$,  such that
\begin{equation}
\mathcal{P}_{\l}(\mathscr{C}^{-}_{-})=(-\infty,0),\qquad \mathcal{P}_{\l}(\mathscr{C}^{-}_{+})=(0,\infty),
\end{equation}
as shown in the right plot of Figure \ref{F2}. It turns out that \eqref{i.1} cannot admit any
negative solution for $\l=0$.
\par
It is worth-emphasizing that, in this case, the algebraic multiplicity of \cite{ELG} equals $2$ and hence, owing to \cite[Th. 5.6.2]{LG01}, the local topological index of $u=0$ does not change as $\l$ crosses the bifurcation value $0$. Consequently, except for the local results of Kielh\"{o}fer \cite{Ki}, no result is available in the literature  to get the global structure of the solution set of \eqref{i.1} bifurcating from $(0,0)$. Note that the global component bifurcating from $(0,0)$ respects \cite[Th. 6.3.1]{LG13}, as the
sum of the parities of its bifurcation points from $u=0$ equals $0$. Actually, according to these results, those loops can only exist when they bifurcate from a point with an even generalized algebraic multiplicity. Otherwise, they should satisfy the global alternative of Rabinowitz \cite{Ra}.
\par
Finally, we assume that  $d>\s_{1}^{-1}$ is sufficiently close to $\s_1^{-1}$. Then, when $q\geq 5$
is an odd integer number, the previous bifurcation diagrams evolve to the global bifurcation diagram plotted in the first picture of Figure \ref{F3}, where the two previous loops emanating from $(0,0)$ separate away from each other generating two compact components,  again denoted by $\mathscr{C}^+$ and $\mathscr{C}^-$,
filled in by positive and negative solutions, respectively, that are separated away from $u=0$. Thus, they are isolas with respect to $u=0$.
\par
Therefore, as $d$ crosses the critical value $\s_1^{-1}$ and $q\geq 5$ is an odd integer, the set of
positive and negative solutions of \eqref{i.1} evolve according to the patterns sketched by the first
plots of Figures \ref{F1}--\ref{F3}, so exhibiting a genuine imperfect bifurcation. In some sense, in this case,  $d=\s_1^{-1}$ can be regarded as a sort of \emph{organizing center} for all admissible bifurcation diagrams of positive and negative solutions of \eqref{i.1}.
\par
As far as it is concerned with the negative solutions of \eqref{i.1} when $d>\s_1^{-1}$ and $q\geq 4$ is an even integer, we were able to prove the existence of a component, $\mathscr{C}^-$, perturbing from the former components $\mathscr{C}^-_{\pm}$ as $d$ perturbs from $\s_1^{-1}$, though it remains an open problem to ascertain whether, or not, $\mathcal{P}_{\l}(\mathscr{C}^{-})=\mathbb{R}$. Moreover, thanks to Lemmas \ref{le5.3} and \ref{le6.3}, it becomes apparent that, as $d$ increases, the $\lambda$-projections of the compact connected components $\mathscr{C}^{+}$, for every $q\geq 4$, and $\mathscr{C}^{-}$, for $q\geq 5$ odd, say
$$
  \mc{P}_{\l}(\mathscr{C}^{+})\equiv[\alpha^{+}(d),\beta^{+}(d)], \quad \mc{P}_{\l}(\mathscr{C}^{-})=[-\beta^{-}(d),-\alpha^{-}(d)],
$$
satisfy
$$
   \lim_{d\ua\infty} \alpha^{\pm}(d)=\infty=\lim_{d\ua \infty} \beta^{\pm}(d).
$$
Therefore, these components move away towards $\pm\infty$ as $d\ua \infty$. 
However, it remains an open problem to ascertain whether the components $\mathscr{C}^{+}$ and $\mathscr{C}^{-}$ diminish shrinking to a single point at some critical $d^*>\s_1^{-1}$ up to disappear for all further values of $d$, or if they are non-empty for all $d>0$.

\begin{center}
	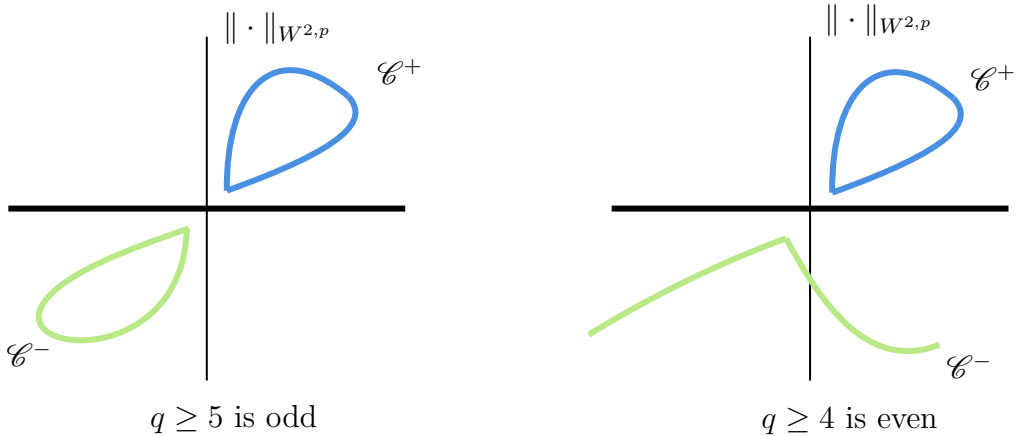
\begin{figure}[h]
	
	\tikzset{every picture/.style={line width=0.75pt}} 
	
	\begin{tikzpicture}[x=0.75pt,y=0.75pt,yscale=-1,xscale=1]
	
	\draw    (179.46,64) -- (179.46,237) ;
	\draw [line width=2.25]    (278.46,150.5) -- (80.46,150.5) ;
	\draw    (480.46,64) -- (480.46,237) ;
	\draw [line width=2.25]    (579.46,150.5) -- (381.46,150.5) ;
	\draw [color={rgb, 255:red, 74; green, 144; blue, 226 }  ,draw opacity=1 ][line width=2.25]    (189.46,141.5) .. controls (189,90) and (213,63) .. (250,94) ;
	\draw [color={rgb, 255:red, 74; green, 144; blue, 226 }  ,draw opacity=1 ][line width=2.25]    (250,94) .. controls (261,106) and (249,120) .. (189.46,141.5) ;
	\draw [color={rgb, 255:red, 184; green, 233; blue, 134 }  ,draw opacity=1 ][line width=2.25]    (468.46,165.5) .. controls (486,198) and (509,233) .. (545,219) ;
	\draw [color={rgb, 255:red, 184; green, 233; blue, 134 }  ,draw opacity=1 ][line width=2.25]    (370,214) .. controls (378,209) and (418,184) .. (468.46,165.5) ;
	\draw [color={rgb, 255:red, 74; green, 144; blue, 226 }  ,draw opacity=1 ][line width=2.25]    (491.46,142.5) .. controls (491,91) and (515,64) .. (552,95) ;
	\draw [color={rgb, 255:red, 74; green, 144; blue, 226 }  ,draw opacity=1 ][line width=2.25]    (552,95) .. controls (563,107) and (551,121) .. (491.46,142.5) ;
	\draw [color={rgb, 255:red, 184; green, 233; blue, 134 }  ,draw opacity=1 ][line width=2.25]    (98,198) .. controls (108,183) and (140,171) .. (170.46,160.5) ;
	\draw [color={rgb, 255:red, 184; green, 233; blue, 134 }  ,draw opacity=1 ][line width=2.25]    (98,198) .. controls (80,225) and (168,231) .. (169.46,160.5) ;
	
	\draw (134,248) node [anchor=north west][inner sep=0.75pt]    {\quad $q\geq 5$ is odd};
	\draw (439,249) node [anchor=north west][inner sep=0.75pt]    {\quad $q\geq 4$ is even};
	\draw (186,50) node [anchor=north west][inner sep=0.75pt]    {$\| \cdot \| _{W{^{2,p}}}$};
	\draw (486,46) node [anchor=north west][inner sep=0.75pt]    {$\| \cdot \| _{W{^{2,p}}}$};
	\draw (263,74) node [anchor=north west][inner sep=0.75pt]    {$\mathscr{C}^{+}$};
	\draw (559,77) node [anchor=north west][inner sep=0.75pt]    {$\mathscr{C}^{+}$};
	\draw (547,222) node [anchor=north west][inner sep=0.75pt]    {$\mathscr{C}^{-}$};
	\draw (78,216) node [anchor=north west][inner sep=0.75pt]    {$\mathscr{C}^{-}$};
	\draw (391,206) node [anchor=north west][inner sep=0.75pt]    {};

	\end{tikzpicture}
	\caption{Two bifurcation diagrams when $d>\s_{1}^{-1}$}
	\label{F3}
	\end{figure}
\end{center}

Under Dirichlet boundary conditions, increasing $d$ promotes a  rapid random movement of the individuals of the species $u$ towards the edges of their territory, $\O$,  where they are washed out by the hostile surroundings. Thus, the positive solutions should become extinct for sufficiently large $d>0$. But the role played  in this model by the parametric transport term $\lambda\langle\mf{a},\nabla u\rangle$ is not well understood yet, and, actually, it might push the individuals towards the interior of the inhabiting area as to avoid extinction. In a rather different context, the extinction for a sufficiently large diffusion coefficient was confirmed, numerically,  in  \cite{LGEMD} (see  \cite[Ch. 2]{LG01}).
\par
Although there is a number of available results concerning the formation of isolas and loops of positive solutions in the context of systems and semilinear elliptic equations (see, e.g., \cite{LGEMD}, \cite{Cano}, \cite{CLGM}, \cite{LGMM},  \cite{KQU} and \cite{FLG}, as well as the references there in), the problem \eqref{i.1} is of a rather different nature, as it inherits a sublinear nature as far it is concerns with
the positive solutions, instead of superlinear indefinite as in most of the references. Moreover, the parameter $\l$ appears incorporated to the differential equation in a rather different way. Actually,
\eqref{i.1} was introduced in \cite{JJ4} as an academic example for testing the abstract theory developed there in. Naturally, the parameter transitions described through this simple example might
enjoy a huge number of applications in applied sciences and engineering. Anyway, up to the best of our knowledge, the transition described by
the first plots of Figures \ref{F1}--\ref{F3} has not been previously described in the literature in
the context of semilinear BVP's.
\par
The fact that $d=\s_1^{-1}$ is a critical value for \eqref{i.1} should not really surprise us because $d\s_1-1$ is the principal eigenvalue of $-d\D - I$ in $\O$ under
Dirichlet boundary conditions. Thus, the stability of zero changes as $d$ crosses $\s_1^{-1}$. Maintaining
the global structure of the solution set when $q\geq 5$ is a more intriguing phenomenology.
\par
This paper is distributed as follows. In Section 2 we collect some preliminaries on the generalized algebraic multiplicity, $\chi$, introduced by Esquinas and L\'{o}pez-G\'{o}mez in  \cite{ELG}, \cite{Es} and \cite{LG01}, and show that any
positive  (resp. negative) solution of \eqref{i.1} is strongly positive (resp. negative). In Section 3, we study the linearization of \eqref{i.1} at $u=0$ to determine the structure of the bifurcation values to positive, or negative, solutions from $u=0$. In Section 4, we analyze the structure of the set of
positive and negative solutions of \eqref{i.1} in a neighborhood of their bifurcations points from $u=0$.
In Section 5 we show the existence of a priori bounds for the positive solutions of \eqref{i.1}. As it is a sublinear problem, these bounds are always available, regardless the value of $q\geq 4$. Things are more challenging concerning the existence of a priori bounds for the negative solutions, because they are given by the positive solutions of a superlinear problem when $q\geq 4$ is even. Thus, in such case, the existence of a priori bounds depends, heavily,  on the size of $q$ and the spatial dimension $N\geq 1$. In Section 6 we will adapt the blowing-up techniques of Gidas and Spruck \cite{GS,GS2}, to get these a priori bounds.
Finally, in Section 7 we will apply the abstract theory developed in \cite{JJ4} to prove the existence of the components $\mathscr{C}^\pm$ and $\mathscr{C}^-_\pm$ already introduced in the description of Figures \ref{F1}--\ref{F3}.
\par
Throughout this paper, for any given pair of real Banach spaces, $U$ and $V$, and any linear continuous operator $T: U\to V$, we will denote by $N[T]$ the null space, or kernel, of $T$, and by $R[T]$ the range, or image, of $T$.

\section{Preliminaries}

\noindent In this section we collect some fundamental properties of the generalized algebraic multiplicity, $\chi$, introduced by Esquinas and L\'{o}pez-G\'{o}mez \cite{ELG}, and later developed in \cite{Es} and \cite{LG01}. This concept is necessary to study the linealization of \eqref{i.1} at $u=0$. Then, we will use the Hopf's maximum principle to show that any positive solution of \eqref{i.1} is strongly positive, and that, similarly,  any negative solution is strongly negative.

\subsection{The generalized algebraic multiplicity}

\noindent Throughout this section, $\mathbb{K}\in\{\mathbb{R},\mathbb{C}\}$, $\Omega$ is a subdomain of $\mathbb{K}$, and, for any given finite dimensional curve $\mathfrak{L}\in\mathcal{C}(\Omega,\mathcal{L}(\mathbb{K}^{N}))$,
a point $\l\in\Omega$ is said to be a \textit{generalized eigenvalue} of $\mathfrak{L}$ if $\mathfrak{L}(\l)\notin GL(\mathbb{K}^{N})$, i.e., $\mathrm{det\,}\mf{L}(\l)=0$. Then, the
\textit{generalized spectrum} of $\mathfrak{L}\in\mathcal{C}(\Omega,\mathcal{L}(\mathbb{K}^{N}))$
is defined by
\begin{equation*}
\Sigma(\mathfrak{L}):=\{\lambda\in\Omega: \mathfrak{L}(\lambda)\notin GL(\mathbb{K}^{N})\}.
\end{equation*}
For analytic curves $\mathfrak{L}\in\mathcal{H}(\O,\mathcal{L}(\mathbb{K}^{N}))$, since $\mathrm{det\,}\mf{L}(\l)$ is analytic in $\l\in\O$, either $\Sigma(\mathfrak{L})=\Omega$,
or $\Sigma(\mathfrak{L})$ is discrete. Thus, $\Sigma(\mf{L})$ consists of isolated generalized
eigenvalues if $\mf{L}(\mu)\in GL(\mathbb{K}^N)$ for some $\mu\in\O$. In such case, the \textit{algebraic multiplicity} of the curve $\mathfrak{L}\in\mathcal{H}(\Omega,\mathcal{L}(\mathbb{K}^{N}))$ at $\lambda_{0}$ is defined through
\begin{equation}
\label{2.1}
\mathfrak{m}_{\alg}[\mathfrak{L},\lambda_{0}]:=\ord_{\lambda_{0}}\det\mathfrak{L}(\lambda).
\end{equation}
This concept extends the classical notion of algebraic multiplicity in linear algebra. Indeed,
if $\mathfrak{L}(\lambda)=\lambda I_{N}-T$ for some linear operator $T\in\mathcal{L}(\mathbb{K}^{N})$, then $\mathfrak{L}\in\mathcal{H}(\mathbb{K},\mathcal{L}(\mathbb{K}^{N}))$ and it is easily seen that $\mathfrak{m}_{\alg}[\mathfrak{L},\lambda_{0}]$ is well defined for all $\l_0\in\Sigma(\mf{L})$ and that
\eqref{2.1} holds. Note that, since $GL(\K^N)$ is open, $I_N-\l^{-1} T\in GL(\K^N)$ for sufficiently large $\l$. Thus, $\l I_N-T\in GL(\K^N)$ and $\Sigma(\mf{L})$ is discrete.
\par
This concept admits a natural (non-trivial) extension to an infinite-dimensional setting. To formalize it, we need to introduce some of notation. In this paper, for any given pair of $\mathbb{K}$-Banach spaces, say $U$ and $V$, we denote by $\Phi_0(U,V)$  the set of linear Fredholm operators of index zero between $U$ and $V$. Then, a \emph{Fredholm (continuous) path,} or curve, is any map $\mathfrak{L}\in \mathcal{C}(\Omega,\Phi_{0}(U,V))$.  Naturally, for any given $\mathfrak{L}\in \mathcal{C}(\Omega,\Phi_{0}(U,V))$, it is said that $\lambda\in\Omega$ is a \emph{generalized eigenvalue} of $\mathfrak{L}$ if $\mathfrak{L}(\lambda)\notin GL(U,V)$, and the \emph{generalized spectrum} of $\mathfrak{L}$, $\Sigma(\mathfrak{L})$,  is defined through   	
\begin{equation*}
\Sigma(\mathfrak{L}):=\{\lambda\in\Omega: \mathfrak{L}(\lambda)\notin GL(U,V)\}.
\end{equation*}
The following concept, going back to \cite{LG01}, plays a pivotal role in the sequel.

\begin{definition}
	\label{de2.1}
	Let $\mathfrak{L}\in \mathcal{C}(\Omega, \Phi_{0}(U,V))$ and $\kappa\in\mathbb{N}$. A generalized eigenvalue $\lambda_{0}\in\Sigma(\mathfrak{L})$ is said to be $\kappa$-algebraic if there exists $\varepsilon>0$ such that
	\begin{enumerate}
		\item[{\rm (a)}] $\mathfrak{L}(\lambda)\in GL(U,V)$ if $0<|\lambda-\lambda_0|<\varepsilon$;
		\item[{\rm (b)}] There exists $C>0$ such that
		\begin{equation}
		\label{2.2}
		\|\mathfrak{L}^{-1}(\lambda)\|<\frac{C}{|\lambda-\lambda_{0}|^{\kappa}}\quad\hbox{if}\;\;
		0<|\lambda-\lambda_0|<\varepsilon;
		\end{equation}
		\item[{\rm (c)}] $\kappa$ is the minimal integer for which \eqref{2.2} holds.
	\end{enumerate}
\end{definition}
Throughout this paper, the set of $\kappa$-algebraic eigenvalues of $\mathfrak{L}$ is  denoted by $\Alg_\kappa(\mathfrak{L})$, and the set of \emph{algebraic eigenvalues} by
\[
\Alg(\mathfrak{L}):=\bigcup_{\kappa\in\mathbb{N}}\Alg_\kappa(\mathfrak{L}).
\]
As in the special case when $U=V=\K^N$, according to Theorems 4.4.1 and 4.4.4 of \cite{LG01}, when $\mathfrak{L}(\lambda)$ is analytic in $\Omega$, i.e., $\mathfrak{L}\in\mathcal{H}(\Omega, \Phi_{0}(U,V))$,  then, either $\Sigma(\mathfrak{L})=\Omega$,
or $\Sigma(\mathfrak{L})$ is discrete and $\Sigma(\mathfrak{L})\subset \Alg(\mathfrak{L})$.
Subsequently, we denote by $\mathcal{A}_{\lambda_{0}}(\Omega,\Phi_{0}(U,V))$ the set  of curves $\mathfrak{L}\in\mathcal{C}^{r}(\Omega,\Phi_{0}(U,V))$ such that $\lambda_{0}\in\Alg_{\kappa}(\mathfrak{L})$ with $1\leq \kappa \leq r$ for some $r\in\mathbb{N}$.
Next, we will construct an infinite dimensional analogue of the classical algebraic multiplicity $\mathfrak{m}_{\alg}[\mathfrak{L},\l_{0}]$ for the class  $\mathcal{A}_{\lambda_{0}}(\Omega,\Phi_{0}(U,V))$. It can be carried out through the theory of Esquinas and L\'{o}pez-G\'{o}mez
\cite{ELG},  where the following pivotal concept, generalizing the transversality condition of
Crandall and Rabinowitz \cite{CR},  was introduced. Throughout this paper, we set
$\mathfrak{L}_{j}:=\frac{1}{j!}\mathfrak{L}^{(j)}(\lambda_{0})$, $1\leq j\leq r$, should these derivatives exist.

\begin{definition}
	\label{de2.3}
	Let $\mathfrak{L}\in \mathcal{C}^{r}(\O,\Phi_{0}(U,V))$ and $1\leq \kappa \leq r$. Then, a given $\lambda_{0}\in \Sigma(\mathfrak{L})$ is said to be a $\kappa$-transversal eigenvalue of $\mathfrak{L}$ if
	\begin{equation*}
	\bigoplus_{j=1}^{\kappa}\mathfrak{L}_{j}\left(\bigcap_{i=0}^{j-1}N(\mathfrak{L}_{i})\right)
	\oplus R(\mathfrak{L}_{0})=V\;\; \hbox{with}\;\; \mathfrak{L}_{\kappa}\left(\bigcap_{i=0}^{\kappa-1}N(\mathfrak{L}_{i})\right)\neq \{0\}.
	\end{equation*}
\end{definition}

For these eigenvalues, the following generalized concept of algebraic multiplicity was introduced by
Esquinas and L\'{o}pez-G\'{o}mez \cite{ELG},
\begin{equation}
\label{ii.3}
\chi[\mathfrak{L}, \lambda_{0}] :=\sum_{j=1}^{\kappa}j\cdot \dim \mathfrak{L}_{j}\left(\bigcap_{i=0}^{j-1}N[\mathfrak{L}_{i}]\right).
\end{equation}
In particular, when $N[\mf{L}_0]=\mathrm{span}[\v_0]$ for some $\v_0\in U$ such that $\mf{L}_1\v_0\notin R[\mf{L}_0]$, then
\begin{equation}
\label{ii.4}
\mf{L}_1(N[\mf{L}_0])\oplus R[\mf{L}_0]=V
\end{equation}
and hence, $\l_0$ is a 1-transversal eigenvalue of $\mf{L}(\l)$ with $\chi[\mf{L},\l_0]=1$. The transversality condition \eqref{ii.4} goes back to Crandall and Rabinowitz \cite{CR}. More generally, under condition \eqref{ii.4},
$$
\chi[\mf{L},\l_0]=\dim N[\mf{L}_0].
$$
According to Theorems 4.3.2 and 5.3.3 of \cite{LG01}, for every $\mathfrak{L}\in \mathcal{C}^{r}(\O, \Phi_{0}(U,V))$, $\kappa\in\{1,2,...,r\}$ and $\lambda_{0}\in \Alg_{\kappa}(\mathfrak{L})$, there exists a polynomial $\Phi: \O\to \mathcal{L}(U)$ with $\Phi(\lambda_{0})=I_{U}$ such that $\lambda_{0}$ is a $\kappa$-transversal eigenvalue of the path
\begin{equation}
\label{ii.5}
\mathfrak{L}^{\Phi}:=\mathfrak{L}\circ\Phi\in \mathcal{C}^{r}(\O, \Phi_{0}(U,V)),
\end{equation}
and $\chi[\mathfrak{L}^{\Phi},\lambda_{0}]$ is independent of the curve of \emph{trasversalizing local isomorphisms} $\Phi$ chosen to transversalize $\mathfrak{L}$ at $\lambda_0$ through \eqref{ii.5}. Therefore, the following concept of multiplicity
is consistent
\begin{equation}
\label{ii.6}
\chi[\mf{L},\l_0]:= \chi[\mathfrak{L}^{\Phi},\lambda_{0}],
\end{equation}
and it can be easily extended by setting
$\chi[\mathfrak{L},\lambda_0] =0$ if $\lambda_0\notin\Sigma(\mathfrak{L})$ and
$\chi[\mathfrak{L},\lambda_0] =+\infty$ if $\lambda_0\in \Sigma(\mathfrak{L})
\setminus \Alg(\mathfrak{L})$ and $r=+\infty$. Thus, $\chi[\mathfrak{L},\lambda]$ is well defined for all  $\lambda\in \O$ of any smooth path $\mathfrak{L}\in \mathcal{C}^{\infty}(\O,\Phi_{0}(U,V))$; in particular, for any analytical curve  $\mathfrak{L}\in\mathcal{H}(\O,\Phi_{0}(U,V))$.
The next uniqueness result, going back to Mora-Corral \cite{MC}, axiomatizes these concepts of algebraic multiplicity. Some refinements of them were delivered in \cite[Ch. 6]{LGMC}. Subsequently, given $\mathfrak{L}, \mathfrak{M} \in \mathcal{C}(\Omega, \Phi_{0}(U))$, we denote by $\mf{L}\circ\mf{M}\in\mathcal{C}(\O,\Phi_{0}(U))$, the curve defined through $[\mf{L}\circ\mf{M}](\l):=\mf{L}(\l)\circ\mf{M}(\l)$ for each $\l\in\O$.

\begin{theorem}
	\label{th24}
	Let $U$ be a $\mathbb{K}$-Banach space. For every $\lambda_{0}\in\mathbb{K}$ and any open neighborhood $\Omega_{\lambda_{0}}\subset\mathbb{K}$ of $\lambda_{0}$, the algebraic multiplicity $\chi$ is the unique map 		 
	\begin{equation*}
	\chi[\cdot, \lambda_{0}]: \mathcal{C}^{\infty}(\Omega_{\lambda_{0}}, \Phi_{0}(U))\longrightarrow [0,\infty]
	\end{equation*}
	such that
	\begin{enumerate}
		\item[{\rm (PF)}] For every pair $\mathfrak{L}, \mathfrak{M} \in \mathcal{C}^{\infty}(\Omega_{\lambda_{0}}, \Phi_{0}(U))$,
		\begin{equation*}
		\chi[\mathfrak{L}\circ\mathfrak{M}, \lambda_{0}]=\chi[\mathfrak{L},\lambda_{0}]+\chi[\mathfrak{M},\lambda_{0}].
		\end{equation*}
		\item[{\rm (NP)}] There exists a rank one projection $\Pi \in \mathcal{L}(U)$ such that
		\begin{equation*}
		\chi[(\lambda-\lambda_{0})\Pi +I_{U}-\Pi,\lambda_{0}]=1.
		\end{equation*}
	\end{enumerate}
\end{theorem}

The axiom (PF) is the  \emph{product formula} and (NP) is a \emph{normalization property}
for establishing the uniqueness of $\chi$. From these two axioms one can derive the remaining properties of  $\chi$; among them, that it equals the classical algebraic multiplicity when
\begin{equation*}
\mathfrak{L}(\lambda)= \lambda I_{U} - K
\end{equation*}
for some compact operator $K$. Indeed, for every $\mathfrak{L}\in \mathcal{C}^{\infty}(\Omega_{\lambda_{0}},\Phi_{0}(U))$, the following properties are satisfied (see \cite{LGMC} for any further details):
\begin{itemize}
	\item $\chi[\mathfrak{L},\lambda_{0}]\in\mathbb{N}\uplus\{+\infty\}$;
	\item $\chi[\mathfrak{L},\lambda_{0}]=0$ if and only if $\mathfrak{L}(\lambda_0)
	\in GL(U)$;
	\item $\chi[\mathfrak{L},\lambda_{0}]<\infty$ if and only if $\lambda_0 \in\Alg(\mathfrak{L})$.
	\item If $U =\mathbb{K}^N$, then, in any basis,
	\begin{equation*}
	\label{1.1.18}
	\chi[\mathfrak{L},\lambda_{0}]= \mathrm{ord}_{\lambda_{0}}\det \mathfrak{L}(\lambda).
	\end{equation*}
	\item For every $K\in \mathcal{K}(U)$ and $\lambda_0\in \s(K)$,
	\begin{equation*}
	\label{1.1.90}
	\chi [\lambda I_U-K,\lambda_{0}]=\mathrm{dim\,}\mathrm{Ker}[(\l_0 I_{U}-K)^{\nu(\l_0)}],
	\end{equation*}
	where $\nu(\l_0)$ is the \emph{algebraic ascent} of $\l_0$, i.e., the minimal integer, $\nu\geq 1$, such that
	\[
	\mathrm{Ker}[(\l_0 I_{U}-K)^{\nu}]=\mathrm{Ker}[(\l_0 I_{U}-K)^{\nu+1}].
	\]
\end{itemize}

\subsection{Strong positivity of the positive solutions of \eqref{i.1}} The change of variable
\begin{equation}
\label{2.7}
u(x)=\zeta(\lambda,d,x) v(x), \qquad \zeta(\lambda,d,x):=e^{-\frac{\lambda}{2d}\langle \mathfrak{a},x\rangle}, \qquad x\in\bar \Omega.
\end{equation}
transforms the problem \eqref{i.1} into the new problem
\begin{equation}
\label{2.8}
\left\{ \begin{array}{ll} -\Delta v= \tfrac{1}{d}\left(1-\tfrac{\lambda^{2} |\mathfrak{a}|^{2}}{4d}\right)v+f_{d}(\lambda,x,v)v & \quad \hbox{in}\;\; \Omega,\\
v =0 & \quad \hbox{on}\;\;\p\O, \end{array}\right.
\end{equation}
where
\begin{equation}
\label{2.9}
   f_{d}(\lambda,x,v):=d^{-1}(\lambda-\zeta^{q-2}(\lambda,d,x)v^{q-2})\zeta(\lambda,d,x)v.
\end{equation}
Thus, if $u\in W_0^{2,p}(\O)$, with $p>N$, is a positive (resp. negative) solution of \eqref{i.1}, then $v$ provides us  with a positive (resp. negative) solution of \eqref{2.8} in $W_0^{2,p}(\O)$. Consequently, the next result holds. Note that, thanks to the Rellich--Kondrachov theorem,   $W^{2,p}(\O)\hookrightarrow \mc{C}^{1,1-\frac{N}{p}}(\bar\O)$ (see, e.g., \cite{GT}, or \cite[Th. 4.5]{LG13}).

\begin{theorem}
\label{th2.4} Any positive solution $u \in W_0^{2,p}(\O)$ of \eqref{i.1} satisfies $u\gg 0$ in the sense that $u(x)>0$ for all $x\in\O$ and $\frac{\p u}{\p n}(x)<0$ for all $x\in\p\O$, where $n$ stands for the outward
unit normal to $\O$ along $\p\O$. Similarly, any negative solution, $w$, satisfies $w\ll 0$ in the sense that $-w\gg 0$.
\end{theorem}
\begin{proof}
Thanks to a result of Bony \cite{Bo}, the Hopf maximum principle,  and the boundary lemma of Hopf--Oleinik work out in the space $W^{2,p}(\O)$ (see, e.g., \cite[Ch.1]{LG13}).
\par
Suppose $u$ is a positive solution of \eqref{i.1}. Then, the function $v$ defined through \eqref{2.7}
is a positive solution of \eqref{2.8} and hence, for some continuous function $b(x)$ whose expression is irrelevant  here, we have that
\begin{equation}
\label{2.10}
  (-\D + b(x))v=0 \quad \hbox{in}\;\;\O.
\end{equation}
Consider a sufficiently large constant, $\o>0$, such that $c:=b+\o\geq 0$ in $\O$. Then,
$$
  (-\D + c)v=\o v \geq 0
$$
and hence, since $\min_{\bar\O}v =0$, it follows from the Hopf maximum principle that $v$ cannot reach
its minimum in $\O$ unless it is constant. Thus, as it cannot be constant, $v(x)>0$ for all $x\in\O$. Moreover, since $\O$ is of class $\mc{C}^2$, by the Hopf--Oleinik boundary lemma, $\frac{\p v}{\p n}(x)<0$ for all $x\in\p\O$. Naturally, the change of variable \eqref{2.7}
preserves these properties. The fact that any negative solution is strongly negative is a direct consequence of the positivity result that we have just proven.
\end{proof}

An alternative proof can be delivered through \cite[Th. 7.10]{LG13}, since \eqref{2.10} entails $v\gneq 0$ to be a principal eigenfunction of $-\D+c$ associated with the eigenvalue $0$. Therefore, $0$ must be the principal eigenvalue of $-\D+c$ and $v\gg 0$ by the Krein--Rutman theorem.

\section{The linearization  of the problem \eqref{i.1} at $u=0$}

\noindent Throughout this paper, we set $\R_+=(0,\infty)$. In this section we study the linealization of \eqref{i.1} at $(\l,d,u)=(\l,d,0)$.
Note that, for every $p>N$,  the solutions of \eqref{i.1} can be regarded as the zeros of the nonlinear operator
$$
   \mathfrak{F}:\mathbb{R}\times \mathbb{R}_{+}\times W^{2,p}_{0}(\Omega)\longrightarrow L^{p}(\Omega)
$$
defined by
\begin{equation}
\label{iii.1}
\mathfrak{F}(\lambda, d, u):=d\Delta u+\l \langle \mathfrak{a},\nabla u\rangle +u+\l  u^{2}-u^{q}, \quad (\lambda, d, u)\in\mathbb{R}\times \mathbb{R}_{+}\times W^{2,p}_{0}(\Omega),
\end{equation}
whose linearization at $(\l, d, u)=(\l,d, 0)$ is given by the linear operator
$$
\mathfrak{L}(\lambda, d):=D_{u}\mathfrak{F}(\lambda,d, 0):
\mathbb{R}\times\mathbb{R}_{+}\longrightarrow\mathcal{L}(W^{2,p}_{0}(\Omega),L^{p}(\Omega))
$$
defined by
\begin{equation}
\label{iii.2}
\mathfrak{L}(\lambda,d)u:=d\Delta u+\lambda \langle \mathfrak{a},\nabla u\rangle+ u, \qquad
(\l,d, u)\in\R \times \mathbb{R}_{+}\times W^{2,p}_{0}(\Omega).
\end{equation}
As for some computations the presence of the gradient term $\nabla u$ in \eqref{iii.1} is somewhat involved, we will perform the change of variables \eqref{2.7}.
To accomplish this task, we introduce the operator surfaces
\begin{align*}
\mathfrak{P} & :\mathbb{R}\times\mathbb{R}_{+}\to GL(W^{2,p}_{0}(\Omega)), \quad
 \mathfrak{P}(\lambda,d)u:=e^{-\frac{\lambda}{2d}\langle \mathfrak{a},x\rangle }u,\\
\mathfrak{R} & :\mathbb{R}\times\mathbb{R}_{+}\to GL(L^{p}(\Omega)), \quad
\mathfrak{R}(\lambda,d)u:=e^{\frac{\lambda}{2d}\langle \mathfrak{a}, x\rangle}u,
\end{align*}
and the associated  linear operator
\begin{equation*}
\mf{C}\in\mathcal{L}(\mathbb{R}\times\mathbb{R}_{+}\times W^{2,p}_{0}(\Omega)), \qquad \mf{C}(\lambda,d,u):=(\lambda,d,\mathfrak{P}(\lambda,d)u).
\end{equation*}
Clearly, $\mf{C}$ is a topological isomorphisms with inverse
\begin{equation*}
\mf{C}^{-1}\in\mathcal{L}(\mathbb{R}\times\mathbb{R}_{+}\times W^{2,p}_{0}(\Omega)), \qquad \mf{C}^{-1}(\lambda,d,u)=(\lambda,d,\mathfrak{R}(\lambda,d)u),
\end{equation*}
and the operator $\mathfrak{N}:\mathbb{R}\times\mathbb{R}_{+}\times W^{2,p}_{0}(\Omega)\to L^{p}(\Omega)$ defined by
$$
  \mathfrak{N}(\l,d,u):=[\mf{R}(\l,d)\circ\mathfrak{F}\circ\mf{C}](\l,d,u)
$$
is given through
\begin{equation}
\label{iii.3}
\mathfrak{N}(\lambda,d,u)=d\Delta u+\left(1-\frac{\lambda^{2}|\mathfrak{a}|^{2}}{4d}\right)u+(\lambda-\zeta^{q-2}(\lambda,d,x)u^{q-2})
\zeta(\lambda,d,x)u^{2}.
\end{equation}
By construction,
\begin{equation*}
\mathfrak{F}^{-1}(0)=\mathfrak{C}(\mathfrak{N}^{-1}(0)).
\end{equation*}
Therefore, the zero sets $\mathfrak{F}^{-1}(0)$ and $\mathfrak{N}^{-1}(0)$ are related via a linear isomorphism. Moreover, since $\mathfrak{C}$ is a positive operator, it preserves the positive (resp. negative) cone of $W^{2,p}_{0}(\Omega)$. As a byproduct, the study of the positive (resp. negative) solutions  of \eqref{i.1} is equivalent to the study of the positive (resp. negative) solutions of the problem
\begin{equation}
\label{iii.4}
\left\{\begin{array}{ll}
-d\Delta u=\left(1-\frac{\lambda^{2}|\mathfrak{a}|^{2}}{4d}\right)u+(\lambda-\zeta^{q-2}(\lambda,d,x)u^{q-2})
\zeta(\lambda,d,x)u^{2}& \text{ in } \Omega, \\
u=0 & \text{ on } \partial\Omega.
\end{array}\right.
\end{equation}
Nevertheless, for some calculations it will be more appropriate using \eqref{i.1} than \eqref{iii.4}.
\par
The next result provides us with the structure of $\mathfrak{L}(\lambda,d)$.

\begin{lemma}
\label{le3.1} $\mathfrak{L}(\lambda,d)\in\Phi_{0}(W^{2,p}_{0}(\Omega),L^{p}(\Omega))$ for each  $(\lambda,d)\in\mathbb{R}\times\mathbb{R}_{+}$.
\end{lemma}
\begin{proof}
Since $-\mathfrak{L}(\lambda,d)$ is uniformly elliptic, there exists a constant $\o(\lambda,d)>0$ such that for every $\gamma\geq \o(\lambda,d)$ and $f\in L^{p}(\Omega)$, the equation
\begin{equation*}
	(-\mathfrak{L}(\lambda,d)+\gamma)v=
   -d\Delta v-\lambda \langle \mathfrak{a},\nabla v\rangle+ (\g-1) v = f
\end{equation*}
has a unique solution $v\in W^{2,p}_{0}(\Omega)$. In other words,
$$
   \mathfrak{L}(\lambda,d)-\gamma J\in GL(W^{2,p}_{0}(\Omega),L^{p}(\Omega)),
$$
where $J$ is the canonical embedding $J:W^{2,p}_{0}(\Omega)\hookrightarrow L^{p}(\Omega)$. Since $J$ is compact and
	\begin{equation*}
	\mathfrak{L}(\lambda,d)=(\mathfrak{L}(\lambda,d)-\gamma J)+\gamma J,
	\end{equation*}	
$\mathfrak{L}(\lambda,d)$ can be expressed as the sum of an isomorphism and a compact operator.
Therefore,  by \cite[Chap. XV, Th. 4.1]{GGS}, the operator $\mathfrak{L}(\lambda,d)$ is Fredholm of index zero.
\end{proof}

Throughout this paper, given a pair of real Banach spaces $(U,V)$ and an operator surface $\mathfrak{M}:\mathbb{R}\times\mathbb{R}_{+}\to\Phi_{0}(U,V)$, $\mf{M}\equiv\mf{M}(\l,d)$, we will denote by $\mf{M}_{d}$ the operator curve given by
\begin{equation*}
	\mf{M}_{d}:\mathbb{R}\to\Phi_{0}(U,V), \quad \mf{M}_{d}(\l):=\mf{M}(\l,d).
\end{equation*}
By $\Sigma_{+}(\mathfrak{M})$ we will denote the subset of $\Sigma(\mathfrak{M})$ consisting of the generalized eigenvalues associated to a positive eigenfunction. The next  result provides us with the structure of $\Sigma_{+}(\mathfrak{L})$, where $\mathfrak{L}(\l,d)$ is the surface defined in \eqref{iii.2}, and shows that $\l_0\in \Sigma_{+}(\mathfrak{L})$ if
$(\l_0,0)$ is a bifurcation point to positive solutions of \eqref{i.1}.

\begin{theorem}
\label{th3.2}
$\Sigma_+(\mathfrak{L})$ is given by
	\begin{equation}
	\label{iii.5}
	\Sigma_{+}(\mathfrak{L})=\{(\l,d)\in\mathbb{R}\times\mathbb{R}_{+}:\; 4d(1-\s_{1}d)=\l^{2}|\mathfrak{a}|^{2}\}
	\end{equation}
(see Figure \ref{F4}). Moreover, the following assertions are satisfied:
\begin{enumerate}
\item[{\rm (i)}] For every $d<\s_{1}^{-1}$, $\Sigma_{+}(\mathfrak{L}_{d})=\{-\lambda_{1}(d),\lambda_{1}(d)\}$, where
\begin{equation*}
		\lambda_{1}(d)=\frac{2}{|\mathfrak{a}|}\sqrt{d(1-d\s_{1})}.
\end{equation*}
Moreover, $\chi[\mathfrak{L}_{d},\pm\lambda_{1}(d)]=1$.
		
\item[{\rm (ii)}] Suppose $d=\s_{1}^{-1}$. Then, $\Sigma_+(\mathfrak{L}_{d})=\{0\}$
and $\chi[\mathfrak{L}_{d},0]=2$.
		
\item[{\rm (iii)}] $\Sigma(\mathfrak{L}_{d})=\emptyset$ if $d>\s_{1}^{-1}$.

\item[{\rm (iv)}] $\l_0\in \Sigma_{+}(\mathfrak{L})$ if
$(\l_0,0)$ is a bifurcation point to positive solutions of \eqref{i.1}.
\end{enumerate}
\end{theorem}

\begin{proof}
Consider the operator surfaces $\mathfrak{P}$ and $\mf{R}$ defined above. Since $\mathfrak{P}(\lambda,d)$ and $\mf{R}(\l,d)$ are topological isomorphisms  for each $(\lambda,d)\in\mathbb{R}\times\mathbb{R}_{+}$, they are in particular Fredholm of index zero, i.e.,
\begin{equation*}
\mathfrak{P}(\lambda,d)\in GL(W^{2,p}_{0}(\Omega))\subset\Phi_{0}(W^{2,p}_{0}(\Omega)), \quad \mf{R}(\l,d)\in GL(L^{p}(\Omega))\subset\Phi_{0}(L^{p}(\Omega)).
\end{equation*}
As the composition of an isomorphism and a Fredholm operator of index zero is again a Fredholm operator of index zero, it follows that
$$
  \mf{R}(\l,d)\circ\mathfrak{L}(\lambda,d)\circ\mathfrak{P}(\lambda,d)
  \in\Phi_{0}(W^{2,p}_{0}(\Omega),L^{p}(\O))
$$
for all $(\lambda,d)\in\mathbb{R}\times\mathbb{R}_{+}$. It is easily seen that
\begin{equation}
\label{iii.6}
[\mf{R}(\l,d)\circ\mathfrak{L}(\lambda,d)\circ\mathfrak{P}(\lambda,d)]u=d\left[\Delta u+\tfrac{1}{d}\left(1-\tfrac{\lambda^{2}|\mathfrak{a}|^{2}}{4d}\right)u\right].
\end{equation}
Thus, since $\mathfrak{P}(\lambda,d)\in GL(W^{2,p}_{0}(\Omega))$ and $\mf{R}(\l,d)\in GL(L^{p}(\O))$, it is apparent that
\begin{equation*}
\Sigma_{+}(\mathfrak{L})=\Sigma_{+}(\mf{R}\circ\mathfrak{L}\circ\mathfrak{P}).
\end{equation*}
Hence, it suffices to find out $\Sigma_{+}(\mf{R}\circ\mathfrak{L}\circ\mathfrak{P})$. From \eqref{iii.6}, it is apparent that $(\lambda,d)\in\Sigma_{+}(\mf{R}\circ\mathfrak{L}\circ\mathfrak{P})$ if and only if
\begin{equation}
\label{iii.7}
\tfrac{1}{d}\left(1-\tfrac{\lambda^{2}|\mathfrak{a}|^{2}}{4d}\right)=\sigma_{1},
\end{equation}
because $\s_1$ is the unique eigenvalue of $-\D$ associated with it there is a positive eigenfunction.
\eqref{iii.5} is a direct consequence of \eqref{iii.7}.
\par
It remains to show the assertions (i)--(iv). Suppose $d<\s_{1}^{-1}$. Then, by \eqref{iii.5}, it is obvious that $\Sigma_{+}(\mathfrak{L}_{d})=\{-\lambda_{1}(d),\lambda_{1}(d)\}$. To compute
$\chi[\mathfrak{L}_{d},\pm \lambda_{1}(d)]$, we use the product formula. According to it,
\begin{equation*}
\chi[\mf{R}_{d}\circ\mathfrak{L}_{d}\circ\mathfrak{P}_{d},\pm\lambda_{1}(d)]
=\chi[\mf{R}_{d},\pm\l_{1}(d)]+\chi[\mathfrak{L}_{d},\pm\lambda_{1}(d)]+\chi[\mathfrak{P}_{d},\pm\lambda_{1}(d)].
\end{equation*}
On the other hand, since $\mathfrak{P}(\lambda,d)\in GL(W^{2,p}_{0}(\Omega))$ and $\mf{R}(\l,d)\in GL(L^{p}(\O))$ for all $(\lambda,d)\in\mathbb{R}\times\mathbb{R}_{+}$, necessarily
$$
  \chi[\mathfrak{P}_{d},\pm\lambda_{1}(d)]=0, \quad \chi[\mf{R}_{d},\pm\l_{1}(d)]=0.
$$
Hence,
$$
   \chi[\mathfrak{L}_{d},\pm\lambda_{1}(d)]=
\chi[\mf{R}_{d}\circ\mathfrak{L}_{d}\circ\mathfrak{P}_{d},\pm\lambda_{1}(d)].
$$
To find out $\chi[\mf{R}_{d}\circ\mathfrak{L}_{d}\circ\mathfrak{P}_{d},\pm\lambda_{1}(d)]$, we denote $\mathfrak{S}:=\mf{R}\circ\mathfrak{L}\circ\mathfrak{P}$. Let $\varphi_{0}\in W^{2,p}_{0}(\Omega)$ be a principal eigenfunction associated with $\s_{1}$. Then, by a direct computation, we find that
\begin{equation*}
N[\mathfrak{S}_{d}(\pm \lambda_{1}(d))]=\mathrm{span\,}[\varphi_{0}], \quad R[\mathfrak{S}_{d}(\pm\lambda_{1}(d))]=\{f\in L^{p}(\Omega): \int_{\Omega} f\varphi_{0} \ dx=0\}.
\end{equation*}
Moreover, differentiating with respect to $\lambda$, yields
\begin{equation*}
\frac{d\mathfrak{S}_{d}}{d\l}(\pm \lambda_{1}(d))u=\mp\frac{\lambda_{1}(d)|\mathfrak{a}|^{2}}{2d}u,
 \qquad u\in W^{2,p}_{0}(\Omega).
\end{equation*}
Thus, the transversality condition
\begin{equation*}
\frac{d\mathfrak{S}_{d}}{d\l}(\pm\lambda_{1}(d))\left( N[\mathfrak{S}_{d}(\pm\lambda_{1}(d))]\right) \oplus R[\mathfrak{S}_{d}(\pm\lambda_{1}(d))]=L^{p}(\Omega)
\end{equation*}
holds. This entails $\pm\lambda_{1}(d)$ to be $1$-transversal eigenvalues. Therefore,
\begin{equation*}
\chi[\mathfrak{L}_{d},\pm\lambda_{1}(d)]=\chi[\mathfrak{S}_{d},\pm\lambda_{1}(d)]=\dim N [\mathfrak{S}_{d}(\pm\lambda_{1}(d))]=1,
\end{equation*}
which ends the proof of Part (i).
\par
Now, suppose $d=\s_{1}^{-1}$. Then, by  \eqref{iii.7}, $\Sigma(\mathfrak{L}_{d})=\{0\}$. Again,
\begin{equation*}
N[\mathfrak{S}_{d}(0)]=\mathrm{span}[\varphi_{0}], \quad R[\mathfrak{S}_{d}(0)]=\{f\in L^{p}(\Omega): \int_{\Omega} f\varphi_{0} \ dx=0\}.
\end{equation*}
However, differentiating with respect to $\lambda$, on this occasion, we find that
\begin{equation*}
\frac{d\mathfrak{S}_{d}}{d\l}(0)u=0, \quad \frac{1}{2!}\frac{d^{2}\mathfrak{S}_{d}}{d\l^{2}}(0)u=-\frac{ |\mathfrak{a}|^{2}}{4d}u, \qquad u\in W^{2,p}_{0}(\Omega).
\end{equation*}
Thus, setting
\begin{equation*}
\mathfrak{S}_{d,0}\equiv\mathfrak{S}_{d}(0), \qquad \mathfrak{S}_{d,1}\equiv \frac{d\mathfrak{S}_{d}}{d\l}(0), \qquad \mathfrak{S}_{d,2}\equiv\frac{1}{2!}\frac{d^{2}\mathfrak{S}_{d}}{d\l^{2}}(0),
\end{equation*}
the following transversality condition holds
\begin{equation*}
\mathfrak{S}_{d,2}(N[\mathfrak{S}_{d,0}]\cap N[\mathfrak{S}_{d,1}])\oplus \mathfrak{S}_{d,1}(N[\mathfrak{S}_{d,0}])\oplus R[\mathfrak{S}_{d,0}]=L^{p}(\Omega),
\end{equation*}
because
$$
  N[\mathfrak{S}_{d,0}]\cap N[\mathfrak{S}_{d,1}]=\mathrm{span}[\varphi_{0}].
$$
Consequently, $\lambda=0$ is a $2$-transversal eigenvalue of $\mathfrak{S}_{d}$ and, due to
\eqref{ii.3},
\begin{equation*}
\chi[\mathfrak{S}_{d},0]=2 \dim \mathfrak{S}_{d, 2}(N[\mathfrak{S}_{d, 0}]\cap N[\mathfrak{S}_{d, 1}])+\dim \mathfrak{S}_{d, 1}(N[\mathfrak{S}_{d, 0}])=2.
\end{equation*}
Hence, $\chi[\mathfrak{L}_{d},0]=\chi[\mathfrak{S}_{d},0]=2$ as stated. This proves Part (ii).
Part (iii) follows directly from \eqref{iii.7}.
\par
To show Part (iv), let $\{(\l_n,u_n)\}_{n\in\N}$ be a sequence of positive solutions of \eqref{iii.4} such that $$
   \lim_{n\to\infty}\l_n=\l_0\quad \hbox{and}\quad \lim_{n\to\infty} u_n=0\quad \hbox{in}\;\; \mc{C}(\bar\O).
$$
Then, setting $\psi_n:=\frac{u_n}{\|u_n\|_\infty}$, $n\geq 1$, we have that, for every $n\geq 1$,
\begin{equation}
\label{iii.8}
- d \psi_n = (-\D)^{-1}\left[ \left(1-\tfrac{\lambda_n^{2}|\mathfrak{a}|^{2}}{4d}\right)\psi_n +(\lambda_n-\zeta^{q-2}(\lambda_n,d,x)u_n^{q-2})
\zeta(\lambda_n,d,x)u_n \psi_n\right]
\end{equation}
in $\O$. By a standard compactness argument, a sequence of $\psi_n$ must approximate  some $\psi_0>0$, which is an eigenfunction associated to $\l_0$. Therefore, $\l_0\in \Sigma_+(\mathfrak{L}_d)$.
\end{proof}

By a rather simple manipulation, it is easily seen that $\Sigma_{+}(\mf{L})$
is the ellipse
\begin{equation*}
\tfrac{\l^{2}}{\alpha^2}+\tfrac{(d-\gamma)^{2}}{\beta^2}=1, \qquad \alpha:=\tfrac{1}{|\mf{a}|\sqrt{\s_{1}}}, \quad \beta:=\tfrac{1}{2\s_{1}}, \quad \gamma:=\tfrac{1}{2\s_{1}},
\end{equation*}
which has been plotted in Figure \ref{F4}.

\begin{center}
	\begin{figure}[h!]

		\tikzset{every picture/.style={line width=0.75pt}} 
		
		\begin{tikzpicture}[x=0.75pt,y=0.75pt,yscale=-1,xscale=1]
		
		\draw    (182,230) -- (416,231) ;
		\draw    (299,47) -- (299,230) ;
		\draw  [dash pattern={on 0.84pt off 2.51pt}]  (416,48) -- (416,231) ;
		\draw  [dash pattern={on 0.84pt off 2.51pt}]  (182,47) -- (182,230) ;
		\draw  [color={rgb, 255:red, 208; green, 2; blue, 27 }  ,draw opacity=1 ][line width=2.25]  (182.25,177) .. controls (182.25,147.73) and (234.52,124) .. (299,124) .. controls (363.48,124) and (415.75,147.73) .. (415.75,177) .. controls (415.75,206.27) and (363.48,230) .. (299,230) .. controls (234.52,230) and (182.25,206.27) .. (182.25,177) -- cycle ;
		\draw  [fill={rgb, 255:red, 0; green, 0; blue, 0 }  ,fill opacity=1 ] (296.5,179.5) .. controls (296.5,178.12) and (297.62,177) .. (299,177) .. controls (300.38,177) and (301.5,178.12) .. (301.5,179.5) .. controls (301.5,180.88) and (300.38,182) .. (299,182) .. controls (297.62,182) and (296.5,180.88) .. (296.5,179.5) -- cycle ;
		\draw    (182,123.5) -- (416,124.5) ;
		
		\draw (152,233) node [anchor=north west][inner sep=0.75pt]    {$-\frac{1}{|\mathfrak{a} |\sqrt{\sigma _{1}}}$};
		\draw (308,38) node [anchor=north west][inner sep=0.75pt]    {$d$};
		\draw (433,223) node [anchor=north west][inner sep=0.75pt]    {$\lambda $};
		\draw (308,160) node [anchor=north west][inner sep=0.75pt]    {$\frac{1}{2\sigma _{1}}$};
		\draw (387,233) node [anchor=north west][inner sep=0.75pt]    {$\frac{1}{|\mathfrak{a} |\sqrt{\sigma _{1}}}$};
		\draw (425,112) node [anchor=north west][inner sep=0.75pt]    {$\sigma _{1}^{-1}$};

		\end{tikzpicture}
	\caption{The spectrum $\Sigma_+(\mathfrak{L}(\l,d))$}
	\label{F4}
	\end{figure}
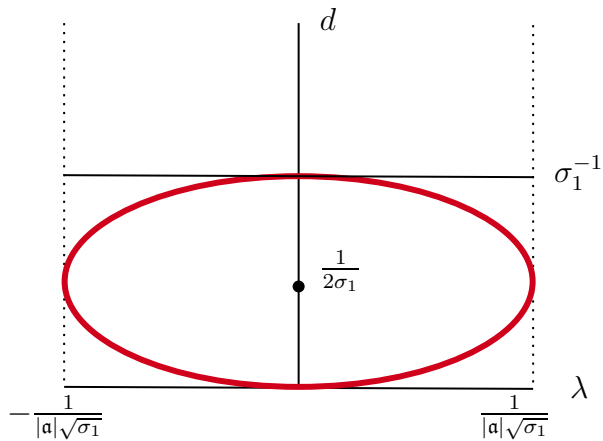
\end{center}

\section{Local structure of the Solution Set when $d\leq \s_1^{-1}$}
\noindent In this section, we study the local structure of the solution set $\mathfrak{F}^{-1}(0)$ in a neighborhood
$$
  (\l,d,u)=(\pm\lambda_{1}(d),d,0)\in\mathbb{R}\times\mathbb{R}_+\times W^{2,p}_{0}(\Omega)
$$
when $d\leq \s_{1}^{-1}$, where we are setting $\l_1(\s_1^{-1})\equiv 0$. This section is divided into two subsections to treat, separately, the cases when $d<\s_{1}^{-1}$ and $d= \s_{1}^{-1}$.
\par
\subsection{The regular case when $d<\s_{1}^{-1}$} Since $\pm\lambda_{1}(d)$ are simple eigenvalues of the curve  $\mathfrak{L}_{d}(\lambda)$ with  $\chi[\mathfrak{L}_{d},\pm\lambda_{1}(d)]=1$, the theorem of
Crandall  and Rabinowitz \cite{CR} provides us with the local estructure of $\mathfrak{F}^{-1}(0)$. Indeed,
setting
\begin{equation*}
Y:=\{f\in L^{p}(\Omega): \int_{\Omega}f\varphi_{0}\ dx=0\},
\end{equation*}
where $\varphi_{0}$ is a principal eigenfunction associated to $\s_1$, and
\begin{equation}
\label{eq4.1}
\mathfrak{F}_{d}(\lambda, u):=d\Delta u+\l \langle \mathfrak{a}, \nabla u\rangle+ u+\l u^{2}- u^{q}, \quad (\lambda, u)\in\mathbb{R}\times W^{2,p}_{0}(\Omega),
\end{equation}
the next result is a direct consequence of \cite{CR}.

\begin{theorem}
	There exist $\varepsilon>0$ and two analytic maps
	\begin{equation*}
	\lambda_{\pm}:(-\varepsilon,\varepsilon)\to \mathbb{R}, \qquad U_{\pm}:(-\varepsilon,\varepsilon)\to Y,
	\end{equation*}
such that  $\lambda_{\pm}(0)=\pm\lambda_{1}(d)$, $U_{\pm}(0)=0$, and, for every $s\in(-\varepsilon,\varepsilon)$,
	\begin{equation*}
	\mathfrak{F}_{d}(\lambda_{\pm}(s),u_{\pm}(s))=0, \quad u_{\pm}(s):=s(\varphi_{0}+U_{\pm}(s)).
	\end{equation*}
	Moreover, there exists $\rho>0$ such that, whenever $\mathfrak{F}_{d}(\lambda,u)=0$ with $(\lambda,u)\in B_{\rho}(\pm\lambda_{1}(d),0)$, either $u=0$,  or $(\lambda,u)=(\lambda_{\pm}(s),u_{\pm}(s))$ for some $s\in(-\varepsilon,\varepsilon)$.
\end{theorem}

Consequently, for every $d<\s_{1}^{-1}$, the set $\mf{F}^{-1}(0)\setminus\{(\l,0)\}$, $\l\sim \pm\l_1(d)$,  consists of two analytic curves
$$
   (\lambda(s),u(s))=(\pm\lambda_{1}(d)+O(1), s\varphi_{0}+O(s)) \quad \text{as } s\to 0
$$
bifurcating from $(\pm\lambda_{1}(d),d,0)\in\mathbb{R}\times W^{2,p}_{0}(\Omega)$, respectively. These solutions are positive if $s>0$, and negative if $s<0$.
\par
Actually, by applying the implicit function theorem to the the operator
$$
  \mathscr{G}(\l,d,y,s):=\left\{ \begin{array}{ll} s^{-1}\mathfrak{F}(\l,d,s(\v_0+y)) & \quad
  \hbox{if}\;\; s\neq 0, \\ \mathfrak{L}(\l,d)(\v_0+y) & \quad \hbox{if}\;\; s=0, \end{array}\right.
$$
at $(\l,d,y,s)=(\pm\l_1(d),d,0,0)$, it becomes apparent that $\mathfrak{F}^{-1}(0)$ consists
of two analytic bi-dimensional manifolds bifurcating from $u=0$ along the curves $\l=\pm\l_1(d)$, $d>0$, $d\sim 0$.

\subsection{The degenerate case when $d=\s_{1}^{-1}$} This case is far more sophisticated, since $\chi[\mathfrak{L}_{d},0]=2$ and hence the theorem of Crandall and Rabinowitz \cite{CR} cannot be applied.
Instead, to ascertain the local structure of $\mathfrak{F}^{-1}(0)$ in this case, we will use some abstract
results  on analytic bifurcation theory, going back to \cite{JJ4} and the references there in,  valid for two arbitrary real Banach spaces, $U$ and $V$, with $U\subset V$,  and any analytic operator $\mathfrak{F}\in\mathcal{H}(\mathbb{R}\times U,V)$ satisfying the following assumptions:
\begin{enumerate}
	\item[(F1)] $\mathfrak{F}(\lambda,0)=0$ for all $\lambda\in\mathbb{R}$.
	\item[(F2)] $D_{u}\mathfrak{F}(\lambda,0)\in\Phi_{0}(U,V)$ for all $\lambda\in\mathbb{R}$.
	\item[(F3)] $N[D_{u}\mathfrak{F}(\lambda_{0},0)]=\mathrm{span}[\varphi_{0}]$ for some $\varphi_0\in U\backslash\{0\}$.
\end{enumerate}
As usual, we denote $ \mathfrak{L}(\lambda):=D_{u}\mathfrak{F}(\lambda,0)$ for all $\lambda\in\mathbb{R}$.
As we are assuming  $\mathfrak{F}$ to be analytic, we can expand it in the form
\begin{equation*}
\mathfrak{F}(\lambda,u)=\mathfrak{L}(\l)u+\sum_{j\geq 0, \; k\geq 2}\lambda^{j}F_{j,k}(u),
\end{equation*}
where $F_{j,k}:U\to V$ are homogeneous polynomials of degree $k\geq 2$ with symmetric polar forms
$$
 F^{(k)}_{j}:U^{k}\to V, \qquad F_{j,k}(u)=F^{(k)}_{j}(u,\cdots,u).
$$
See \cite[Chap. 26]{HP} for the definition of this concepts. Throughout this section, $\langle \cdot,\cdot\rangle:V\times V'\to\mathbb{R}$ denotes the duality pairing between $V$ and its topological dual space $V'$. By (F3), we can normalize $\langle \varphi_{0},\varphi_{0}^{\ast}\rangle=1$, where $\varphi_{0}^{\ast}$ spans the null space of the
adjoint operator $\mathfrak{L}^{\ast}_{0}:V'\to U'$. Subsequently, we consider the projection operators
\begin{align*}
P: & \; V\to N[\mathfrak{L}_{0}], \quad P(v):=\langle v,\varphi_{0}^{\ast}\rangle \varphi_{0},\\
Q: & \; V\to R[\mathfrak{L}_0], \quad Q(v):=v-P(v),
\end{align*}
and identify $\mathbb{R}\times N[\mathfrak{L}_{0}]$ with $\mathbb{R}^{2}$ via the isomorphism $T(\lambda,x\varphi_{0})=(\lambda,x)$. By a standard Lyapunov--Schmidt reduction (see \cite{JJ3} and \cite[Ch. 3]{LG01}, if necessary),  there exist an open neighborhood $\mathcal{U}$ of $(\lambda_{0},0)$ in $\mathbb{R}^{2}$, an analytic map $\mathcal{Y}:T^{-1}(\mathcal{U})\to V$, and a finite dimensional operator
\begin{equation*}
	\mathfrak{G}: \mathcal{U}\subset\mathbb{R}^{2} \to \mathbb{R}, \qquad \mf{G}(\lambda,x):=(I_{V}-Q)\mathfrak{F}(\lambda,x\varphi_{0}+\mathcal{Y}(\lambda,x\varphi_{0})),
\end{equation*}
such that $(\lambda,x)\in \mathcal{U}$ satisfies $\mathfrak{G}(\lambda,x)=0$ if and only if $(\lambda,u)=(\lambda,x\varphi_{0}+\mathcal{Y}(\lambda,x\varphi_{0}))$ satisfies $\mathfrak{F}(\lambda,u)=0$. Actually, there exists an open neighborhood $\mathcal{V}$ of $(\lambda_{0},0)$ in $\mathbb{R}\times U$ such that
\begin{align*}
\psi:&\;\mathfrak{F}^{-1}(0)\cap \mathcal{V}\longrightarrow \mathfrak{G}^{-1}(0)\cap\mathcal{U}, \quad (\lambda,u)\mapsto (\lambda,\langle u, \varphi_{0}^{\ast}\rangle),\\
\psi^{-1}:&\; \mathfrak{G}^{-1}(0)\cap \mathcal{U}\longrightarrow \mathfrak{F}^{-1}(0)\cap\mathcal{V}, \quad (\lambda,x)\mapsto (\lambda,x\varphi_{0}+\mathcal{Y}(\lambda,x\varphi_{0})),
\end{align*}
are analytic and mutual inverses.  Therefore, the analytical structures of $\mathfrak{F}^{-1}(0)\cap \mathcal{V}$ and $\mathfrak{G}^{-1}(0)\cap\mathcal{U}$ coincide. In particular, $(\lambda,x)$ is a regular point of $\mathfrak{G}$, i.e., $D_{x}\mathfrak{G}(\l,x)\neq 0$, if and only if, $(\lambda,x\varphi_{0}+\mathcal{Y}(\lambda,x\varphi_{0}))$ is a regular point of $\mathfrak{F}$, i.e.,
$$
 D_{u}\mathfrak{F}(\lambda,x\varphi_{0}+\mathcal{Y}(\lambda,x\varphi_{0}))\in GL(U,V).
$$
In other words, both the regular and the singular points are preserved.
\par
Without loss of generality, we can assume that $(\lambda_{0},0)=(0,0)$. In this way, the infinite dimensional problem can be reduced, locally, to the finite dimensional problem $\mathfrak{G}(\lambda,x)=0$, with $(\lambda,x)\in\mathcal{U}\subset\mathbb{R}^{2}$. Since $\mathfrak{G}$ is analytic and $\mathfrak{G}(\lambda,0)=0$, it admits the local expansion
\begin{equation*}
\mathfrak{G}(\lambda,x)=\sum_{i\geq 0,\; j\geq 1} a_{ij}\lambda^{i} x^{j}, \qquad (\lambda,x)\sim (0,0),
\end{equation*}
for certain coefficients $a_{ij}\in\mathbb{R}$, $(i,j)\in\mathbb{Z}_{+}^{2}$, $j\neq 0$. Thus, there exists an analytic function $g:\mathcal{U}\to \mathbb{R}$ such that
\begin{equation*}
\mathfrak{G}(\lambda,x)=x\sum_{i\geq 0,\; j\geq 1} a_{ij}\lambda^{i} x^{j-1}=xg(\lambda,x), \qquad (\lambda,x)\sim (0,0).
\end{equation*}
According to \cite[Sect. 6]{JJ4},  it follows from (F3)  that
\begin{equation}
\label{nueva}
1\leq \chi\equiv \chi[\mathfrak{L},0]=\ord_{\lambda=0}D_{x}\mathfrak{G}(\lambda,0)=\ord_{\lambda=0} g(\lambda,0).
\end{equation}
Hence, $g:\mathcal{U}\to\mathbb{R}$ can be expanded in the form
\begin{equation}
\label{iv.2}
g(\lambda,x)=\sum_{\nu=0}^{s}C_{\nu}\lambda^{j_{\nu}}x^{l_{\nu}}+\sum_{j,k}C_{j,k}\lambda^{j}x^{k},
\end{equation}
where $(\ell_{0},j_{0})=(0,\chi)$, $\chi>j_{1}>\cdots>j_{s}$, $0<\ell_{1}<\cdots<\ell_{s}$, and the summation of the second sum is taken only on the points $(k,j)$ lying above the polygonal line joining $(0,\chi)$, $(\ell_1,j_1)$, $\cdots$, $(l_{s},j_{s})$, or on the line $j=j_{s}$. The polygonal line joining the points $(0,\chi)$, $(\ell_{1},j_{1})$, $\cdots$, $(\ell_{s},j_{s})$ is usually called the Newton's polygon of $g$. Subsequently, we will use the next result of Kielh\"ofer \cite{Ki}. It is rewritten with our own notations here. Actually, the last assertion is based on \cite[Th. 4.3.3]{LGMC}.

\begin{theorem}
\label{th4.2}
	Let $\mathfrak{F}:\mathbb{R}\times U\to V$ be an analytic operator satisfying hypothesis {\rm (F1)--(F3)} with $\chi[\mathfrak{L},0]=\chi\geq 1$ and having the expansion  	
\begin{equation}
\label{iv.3}	 \mathfrak{F}(\lambda,u)=\mathfrak{L}(\lambda)u+\sum_{\nu=1}^{s}\lambda^{j_{\nu}}F_{j_{\nu},\ell_{\nu}+1}(u)
+\sum_{j,k}\lambda^{j}F_{j,k+1}(u), \quad \lambda\in \mathbb{R}, \ u\in U,
\end{equation}
where $\chi>j_{1}>\cdots>j_{s}$,  $0<\ell_{1}<\cdots<\ell_{s}$, and the summation of the second sum is taken on points $(k,j)$ lying above the polygonal line joining $(0,\chi)$, $(\ell_1,j_1)$, $\cdots$, $(\ell_{s},j_{s})$, or on the line $j=j_{s}$. If, in addition,
$$
   H_{\nu}:=\langle F_{j_{\nu}}^{(l_{\nu}+1)}(\varphi_{0},\cdots,\varphi_{0}),\varphi_0^{\ast}\rangle\neq 0, \quad \nu=1,\cdots,s,
$$
then the Newton polygon associated to the reduced map $g:\mathcal{U}\to\mathbb{R}$ defined by  \eqref{iv.2}  is given by the lower convex hull of the points $(0,\chi)$, $(\ell_1,j_1)$, $\cdots$, $(\ell_{s},j_{s})$. Furthermore, the corresponding coefficients are precisely the numbers
$$
  C_{\nu}=H_{\nu}, \quad \nu=1,\cdots,s,
$$
and $C_{0}=\rho^{(\chi)}(0)$, where $\rho(\lambda)$ is the perturbed eigenvalue from $0$ of the operator $\mathfrak{L}(\lambda)$.
\end{theorem}
\par
Next, we apply Theorem \ref{th4.2} with $d=\s_{1}^{-1}$ to the operator $\mathfrak{F}_{d}:\mathbb{R}\times W^{2,p}_{0}(\Omega)\to L^{p}(\Omega)$ defined by
$$
   \mathfrak{F}_{d}(\lambda,u):=d\Delta u+\lambda \langle \mathfrak{a},\nabla u\rangle
   + u+\lambda u^{2}-u^{q},\quad (\l,u)\in \mathbb{R}\times W^{2,p}_{0}(\Omega).
$$
Since $q$ is integer, $\mathfrak{F}_d$ is analytic. Moreover, by our previous analysis,
it satisfies hypothesis (F1)-(F3) with $\l_{0}=0$ and $\varphi_{0}$ the principal eigenfunction associated to $\s_{1}$. As in this setting $V=L^{p}(\Omega)$, we have that $V' =L^{p'}(\Omega)$, where
$p'$ is  the H\"{o}lder conjugate of $p$, i.e.,
$$
  \frac{1}{p}+\frac{1}{p'}=1.
$$
Thus,  the duality pairing $\langle\cdot,\cdot\rangle_{V,V'}$ is given through
\begin{equation*}
\langle f,g \rangle \equiv \langle f, g \rangle_{V,V'}:=\int_{\Omega}fg \, dx, \qquad (f,g)\in L^{p}(\Omega)\times L^{p'}(\Omega).
\end{equation*}
In order to apply Theorem \ref{th4.2}, it is appropriate to express the operator $\mathfrak{F}_{d}$ in the form
\begin{equation*}
\mathfrak{F}_{d}(\lambda,u)=\mathfrak{L}_{d}(\lambda)u+\lambda F_{1,2}(u)+F_{0,q}(u), \qquad (\lambda,u)\in \mathbb{R}\times W^{2,p}_{0}(\Omega),
\end{equation*}
where
\begin{equation}
\label{iv.4}
\mathfrak{L}_{d}(\lambda)u:=d\Delta u+\lambda \langle \mathfrak{a},\nabla u\rangle+ u, \qquad
(\l, u)\in\R \times W^{2,p}_{0}(\Omega),
\end{equation}
and
\begin{equation*}
F_{1,2}(u):=u^{2}, \quad F_{0,q}(u):=-u^{q}, \qquad u\in W^{2,p}_{0}(\O),
\end{equation*}
which is consistent with the notations used in the expansion \eqref{iv.3}.
\par
According to Theorem \ref{th3.2} (ii), $\chi[\mathfrak{L}_{d},0]=2$. Thus, Theorem \ref{th4.2} implies that
\begin{equation}
\label{iv.5}
g(\lambda,x)=C_{0}\lambda^{2}+\langle \varphi_{0}^{2},\varphi_{0}\rangle x\lambda-\langle \varphi_{0}^{q},\varphi_{0}\rangle x^{q-1}+\sum_{j,k}C_{j,k}\lambda^{j}x^{k}, \quad (\lambda,x)\in\mathcal{U},
\end{equation}
where the summation of the second sum is only over points $(k,j)$ lying above the polygonal with vertices  $(0,2)$, $(1,1)$ and $(q-1,0)$. The Newton polygon of $g$ is the convex hull of the points $(0,2)$, $(1,1)$ and $(q-1,0)$ represented in Figure \ref{F5}.

\begin{center}
	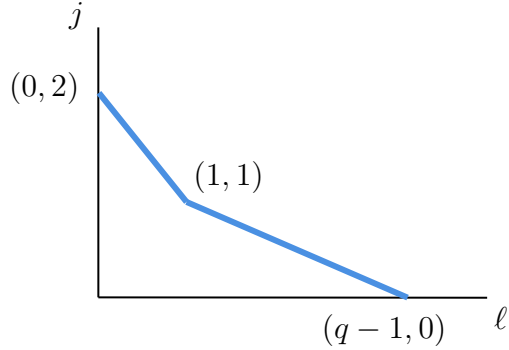
\begin{figure}[h!]
		\tikzset{every picture/.style={line width=0.75pt}} 
		
		\begin{tikzpicture}[x=0.75pt,y=0.75pt,yscale=-1,xscale=1]
		
		\draw    (156.98,59.19) -- (157,195) ;
		\draw    (157,195) -- (350.98,195.19) ;
		\draw [color={rgb, 255:red, 74; green, 144; blue, 226 }  ,draw opacity=1 ][line width=2.25]    (157.28,92.13) -- (201,147) ;
		\draw [color={rgb, 255:red, 74; green, 144; blue, 226 }  ,draw opacity=1 ][line width=2.25]    (201,147) -- (311.28,195.13) ;
		
		\draw (203,125) node [anchor=north west][inner sep=0.75pt]    {$( 1,1)$};
		\draw (111,80) node [anchor=north west][inner sep=0.75pt]    {$( 0,2)$};
		\draw (267,202) node [anchor=north west][inner sep=0.75pt]    {$( q-1,0)$};
		\draw (141,45) node [anchor=north west][inner sep=0.75pt]    {$j$};
		\draw (352.98,198.19) node [anchor=north west][inner sep=0.75pt]    {$\ell$};

		\end{tikzpicture}
		\caption{Newton diagram of $g(\l,x)$}
		\label{F5}
	\end{figure}
\end{center}

The next result follows readily from the last assertion of Theorem \ref{th4.2}.

\begin{lemma}
	The first coefficient of $g(\lambda,x)$ is given by $C_{0}=-\frac{|\mathfrak{a}|^{2}}{2d}$.
\end{lemma}

\begin{proof}
Indeed, since $\chi[\mathfrak{L}_{d},0]=2$, by Theorem \ref{th4.2}, the coefficient $C_{0}$ is given by $\rho^{(2)}(0)$ where $\rho(\lambda)$ is the perturbation of the zero eigenvalue of $\mathfrak{L}_{d}(\l)$.
By definition, $\rho(\lambda)$ satisfies the eigenvalue problem
\begin{equation}
\label{iv.6}
	\left\{ \begin{array}{ll} d\Delta u+\lambda  \langle \mathfrak{a},\nabla u\rangle+u=\rho(\lambda)u &
\quad \text{ in } \Omega,\\ u=0 & \quad \text{ on } \partial\Omega.\end{array}\right.
\end{equation}
As the change of variables $u(x)=e^{-\frac{\lambda }{2d}\langle \mathfrak{a},x\rangle}v(x)$, $x\in\Omega$, transforms \eqref{iv.6} into
\begin{equation*}
\left\{ \begin{array}{ll} \Delta v+\left[\frac{1}{d}\left(1-\frac{\lambda^{2}|\mathfrak{a}|^{2}}{4d}\right)-\frac{1}{d}\rho(\lambda)\right]v=0
 & \quad \text{ in } \Omega,\\ v=0 & \quad   \text{ on } \partial\Omega,\end{array}\right.
\end{equation*}
and we are assuming that  $1/d=\s_{1}$, it becomes apparent that  $\rho(\lambda)=-\frac{\lambda^{2}|\mathfrak{a}|^{2}}{4d}$. Therefore,
$C_{0}=\rho^{(2)}(0)=-\frac{|\mathfrak{a}|^{2}}{2d}$.  This ends the proof.
\end{proof}

Substituting the value $C_0=-\frac{|\mathfrak{a}|^{2}}{2d}$ in \eqref{iv.5}, and owing the
Newton--Puiseux algorithm, we obtain the following asymptotic expansion of solutions of $g(\lambda,x)=0$ in a neighborhood of $(0,0)$,
\begin{equation}
\label{iv.7}
\begin{split}
x(\lambda) & =\tfrac{|\mathfrak{a}|^{2}}{2d\langle \varphi_{0}^{2},\varphi_{0}\rangle}\lambda+O(\lambda)
\quad \hbox{as}\;\;\l\to 0, \\
 x(\lambda) & =\pm \left(\tfrac{\langle \varphi_{0}^{2},\varphi_{0}\rangle}{\langle \varphi_{0}^{q},\varphi_{0}\rangle}\right)^{\frac{1}{q-2}}\lambda^{\frac{1}{q-2}}
 +O(\lambda^{\frac{1}{q-2}})\quad \hbox{as}\;\; \l\downarrow 0,
\end{split}
\end{equation}
if $q$ is even,  and
\begin{equation}
\label{iv.8}
\begin{split}
x(\lambda) & =\tfrac{|\mathfrak{a}|^{2}}{2d\langle \varphi_{0}^{2},\varphi_{0}\rangle}
\lambda+O(\lambda) \quad \hbox{as}\;\;\l\to 0, \\
x(\lambda) & =\left(\tfrac{\langle \varphi_{0}^{2},\varphi_{0}\rangle}{\langle \varphi_{0}^{q},\varphi_{0}\rangle}\right)^{\frac{1}{q-2}}\lambda^{\frac{1}{q-2}}
+O(\lambda^{\frac{1}{q-2}})\quad \hbox{as}\;\; \l\to 0,
\end{split}
\end{equation}
if $q$ is odd. By Theorem \ref{th3.2}(ii), $\chi[\mathfrak{L}_{d},0]=2$. Thus,
it follows from \eqref{nueva} that
\begin{equation*}
\chi[\mathfrak{L}_{d},0] = \ord_{\l=0}g(\l,0)=2.
\end{equation*}
Consequently, by the Weierstrass--Malgrange preparation theorem, shortening the neighborhood $\mathcal{U}=\mathcal{U}_{\lambda}\times\mathcal{U}_{x}\subset\mathbb{R}^{2}$,  if necessary,  there exists an analytic function $c:\mathcal{U}\to\mathbb{R}$ such that  $c(0,0)\neq 0$, plus $\chi=2$ analytic functions, $c_{j}:\mathcal{U}_{x}\to\mathbb{R}$, $c_{j}(0)=0$, $j=1,2$, such that
\begin{equation*}
g(\lambda,x)=c(\lambda,x)\left[\lambda^{2}+c_{1}(x)\lambda+c_{2}(x)\right].
\end{equation*}
Hence for each $x\in \mathcal{U}_{x}$, the equation $g(\lambda,x)=0$ has, at most, two solutions. This shows that indeed, the solutions \eqref{iv.7} and \eqref{iv.8}, are the unique ones of $g(\lambda,x)=0$ in a neighbourhood of $(0,0)$. Figure \ref{F} represents these branches in each of these cases.
\par It should be noted that if $q=3$, the Newton diagram illustrated in Figure \ref{F5} becomes an straight line. Consequently, the Newton--Puiseux algorithm does not provide, in general, with two branches of solutions. Actually, a direct application of the Newton--Puiseux algorithm shows that the number of branches depends on the real solutions $y$ of the polynomial
\begin{equation*}
\langle\varphi_{0}^{3},\varphi_{0}\rangle y^{2}-\langle\varphi_{0}^{2},\varphi_{0}\rangle y+\frac{|\mf{a}|^{2}}{2d}=0.
\end{equation*}
As the special case $q=3$ does not follow the patterns of the general case $q\geq 4$, it has been left outside the general scope of this paper. Thus, throughout the rest of this paper we assume that $q\geq 4$.

\begin{center}
	\begin{figure}[h!]

		\tikzset{every picture/.style={line width=0.75pt}} 
		
		\begin{tikzpicture}[x=0.75pt,y=0.75pt,yscale=-1,xscale=1]
		
		\draw    (287.04,25.78) -- (287,147) ;
		\draw [line width=2.25]    (345.04,86.28) -- (229,86.5) ;
		\draw    (99.04,25.78) -- (99,147) ;
		\draw [line width=2.25]    (157.04,86.28) -- (41,86.5) ;
		\draw [color={rgb, 255:red, 74; green, 144; blue, 226 }  ,draw opacity=1 ][line width=2.25]    (338.33,58.18) -- (287.02,86.88) ;
		\draw [color={rgb, 255:red, 74; green, 144; blue, 226 }  ,draw opacity=1 ][line width=2.25]    (150.33,58.18) -- (99.02,86.39) ;
		\draw [color={rgb, 255:red, 74; green, 144; blue, 226 }  ,draw opacity=1 ][line width=2.25]    (99.02,86.39) .. controls (116.02,52.39) and (149,48.51) .. (150.02,47.39) ;
		\draw [color={rgb, 255:red, 126; green, 211; blue, 33 }  ,draw opacity=1 ][line width=2.25]    (52,127) .. controls (77,119) and (89,108) .. (99.02,86.39) ;
		\draw [color={rgb, 255:red, 126; green, 211; blue, 33 }  ,draw opacity=1 ][line width=2.25]    (99.02,86.39) -- (47.71,114.6) ;
		\draw [color={rgb, 255:red, 126; green, 211; blue, 33 }  ,draw opacity=1 ][line width=2.25]    (287.02,86.88) -- (235.7,115.09) ;
		\draw [color={rgb, 255:red, 74; green, 144; blue, 226 }  ,draw opacity=1 ][line width=2.25]    (287.02,86.39) .. controls (291.29,45.92) and (339.59,46.77) .. (341.29,46.92) ;
		\draw [color={rgb, 255:red, 126; green, 211; blue, 33 }  ,draw opacity=1 ][line width=2.25]    (287.02,86.88) .. controls (291.29,120.92) and (331.59,119.77) .. (340.59,118.77) ;
		
		\draw (244,157) node [anchor=north west][inner sep=0.75pt]    {\quad $q\geq 4$ \hbox{ is even}};
		\draw (54,157) node [anchor=north west][inner sep=0.75pt]    {\quad $q\geq 5$ \hbox{ is odd}};

		\end{tikzpicture}
	\caption{$g^{-1}(0)$ in a neighborhood of $(0,0)$}
	\label{F}
	\end{figure}
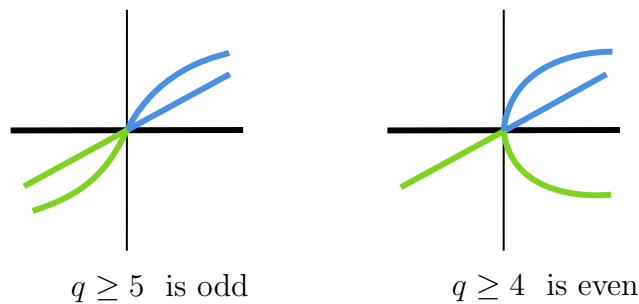
\end{center}

The following result establishes that in a neighborhood of $(\l,u)=(0,0)$ the solutions of \eqref{i.1}, $(\l,u)$, must be either positive or negative. By Theorem \ref{th2.4}, this entails $u\gg 0$, or $u\ll 0$.

\begin{proposition}
\label{pr4.4}
There exists $\varepsilon>0$ such that either $u\gg 0$, or $-u\gg 0$, for every
$(\lambda,u)\in\mathfrak{F}_{d}^{-1}(0)$ such that $|\lambda|+\|u\|_{W^{2,p}}<\varepsilon$.
In other words, there exists a neighborhood of $(0,0)$ in  $\mathbb{R}\times W^{2,p}_{0}(\Omega)$,  $\mathcal{U}$, such that $\mathfrak{F}_{d}^{-1}(0)\cap \mathcal{U}$ consists of positive, or negative,  solutions.
\end{proposition}
\begin{proof}
Let $\{(\lambda_{n},u_{n})\}_{n\in\mathbb{N}}\subset \mathfrak{F}_{d}^{-1}(0)$ be a sequence of solutions such that
\begin{equation}
\label{iv.9}
\lim_{n\to\infty}(\lambda_{n},u_{n})=(0,0) \quad \hbox{in}\;\; \mathbb{R}\times W^{2,p}_{0}(\Omega).
\end{equation}
Since $W^{2,p}_{0}(\Omega)\hookrightarrow\mathcal{C}^{1,1-\frac{N}{p}}(\bar \Omega)$, \eqref{iv.9} holds in $\mathbb{R}\times\mathcal{C}^{1}(\overline{\Omega})$. Performing the change of variables
\begin{equation*}
  v_{n}(x):=e^{\frac{\lambda}{2d}\langle \mathfrak{a},  x\rangle}u_{n}(x), \quad x\in\Omega\subset\mathbb{R}^{N},
\end{equation*}
it is apparent that, for every $n\geq 1$,
\begin{equation*}
v_{n}=(-\Delta)^{-1}\left[\tfrac{1}{d}\left(1-\tfrac{\lambda_{n}^{2} |\mathfrak{a}|^{2}}{4d}\right)v_{n}+f_{d}(\lambda_{n},x,v_{n})v_{n}\right],
\end{equation*}
where, setting $\zeta(\lambda,d,x) :=e^{-\frac{\lambda}{2d}\langle \mathfrak{a},x\rangle}$,
we have denoted
$$
    f_{d}(\lambda,x,v) :=d^{-1}(\lambda-\zeta^{q-2}(\lambda,d,x)v^{q-2})\zeta(\lambda,d,x)v
$$
(see \eqref{iii.4}, if necessary). Then, the functions $\psi_{n}:=\frac{v_{n}}{\|v_{n}\|_{\infty}}$, $n\geq 1$, satisfy $\|\psi_{n}\|_{\infty}=1$ and
\begin{equation}
\label{iv.10}
\psi_{n}=(-\Delta)^{-1}\left[\tfrac{1}{d}\left(1-\tfrac{\lambda_{n}^{2} |\mathfrak{a}|^{2}}{4d}\right)\psi_{n}+f_{d}(\lambda_{n},x,v_{n})\psi_{n}\right]
\end{equation}
for all $n\geq 1$. On the other hand, the sequence
\begin{equation*}
g_{n} :=\tfrac{1}{d}\left(1-\tfrac{\lambda_{n}^{2} |\mathfrak{a}|^{2}}{4d}\right)\psi_{n}+f_{d}(\lambda_{n},x,v_{n})\psi_{n}, \quad n\geq 1,
\end{equation*}
is bounded in $\mc{C}(\bar \O)$. Thus, by the compactness of $(-\Delta)^{-1}$, there exists $\psi\in\mathcal{C}(\bar {\Omega})$ such that, along some subsequence, relabeled by $n\geq 1$,
$\lim_{n\to\infty}\psi_{n}=\psi$  in $\mathcal{C}(\bar {\Omega})$. By \eqref{iv.9},
\begin{equation*}
\lim_{n\to\infty} f_{d}(\lambda_{n},x,v_{n})=0.
\end{equation*}
Thus, letting $n\to\infty$ in \eqref{iv.10}, yields $\psi=d^{-1}(-\Delta)^{-1}\psi$. Therefore,
since $\sigma_{1}=d^{-1}$, by the simplicity of $\s_1$, it is apparent that $\psi=\pm\varphi_{0}$.
Hence, either $\psi\gg 0$, or $\psi\ll 0$ (i.e., $-\psi\gg 0$). This entails that either
 $v_n\gg 0$, or $v_n\ll 0$, for sufficiently large $n$, and ends the proof.
\end{proof}

In the previous Lyapunov--Schmidt reduction we have denoted $x=\langle u, \varphi_{0}\rangle$. Thus,
\begin{equation*}
x=\langle u, \varphi_{0}\rangle=\int_{\Omega}u\varphi_{0}\, dx.
\end{equation*}
Thus, by Proposition \ref{pr4.4}, the solutions of the bifurcation equation
$$
   \mf{G}_{d}(\lambda,x)=(I_{V}-Q)\mathfrak{F}_{d}(\lambda,x\varphi_{0}+\mathcal{Y}(\lambda,x\varphi_{0}))=0
$$
are positive for $x>0$ and negative for $x<0$. Therefore, according to the asymptotic expansions \eqref{iv.7} and \eqref{iv.8}, when $q$ is even, for $\l>0$ there emanate from $(0,0)$ two branches of positive solutions and one branch of negative solutions, while another branch of negative solutions emanates for $\l<0$, as illustrated by Figure \ref{F}. Similarly, when $q$ is odd, for $\l>0$ there emanate from $(0,0)$ two branches of positive solutions, while there  emanate another two branches of negative solutions  for $\l<0$, as illustrated by Figure \ref{F}. This concludes the analysis of the local structure of the solution set when $d=\s_1^{-1}$.

\section{A priori bounds for the positive Solutions}

\noindent In this section we establish the existence of a priori bounds for the positive solutions of  \eqref{i.1}. For any given $d>0$, we will denote by  $\mathscr{S}_{d}$ the set of positive solutions of \eqref{i.1}, i.e.,
\begin{equation*}
\mathscr{S}_{d}:=\{(\l,u)\in\mathfrak{F}_{d}^{-1}(0): \;u\gg 0\}\subset \mathbb{R}\times W^{2,p}_{0}(\Omega).
\end{equation*}
The next result shows the existence of a priori bounds regardless the size of $d>0$.

\begin{lemma}
	\label{le5.1}
There exists a real valued function $C:\mathbb{R}\to(0,+\infty)$, such that, for every
$d>0$ and  $(\lambda,u)\in\mathscr{S}_{d}$,
\begin{equation}
\label{v.1}
	\|u\|_{\infty}\leq C(\l)\leq \left\{\begin{array}{ll} 1 & \quad \hbox{if}\;\; \lambda\leq 0, \\
    \lambda^{\frac{1}{q-2}}+1, & \quad \hbox{if}\;\; \lambda > 0. \end{array}\right.
\end{equation}
\end{lemma}
\begin{proof}
Let $(\lambda,u)\in\mathscr{S}_{d}$ be a positive solution of \eqref{i.1}, and $x_{0}\in\Omega$ such that
$u(x_{0})=\|u\|_{\infty}$. Since $u\in W^{2,p}_{0}(\Omega)$ with $p>N$, the maximum principle of Bony
\cite{Bo} entails that $\nabla u(x_{0})=0$ and $\Delta u(x_{0})\leq 0$. Thus,
	\begin{align*}
	0\leq  -d\Delta u(x_{0}) & =\lambda \langle \mathfrak{a},\nabla u(x_{0})\rangle +u(x_{0})+\lambda u^{2}(x_{0})-u^{q}(x_{0}) \\ 	&=u(x_{0})+\lambda u^{2}(x_{0})-u^{q}(x_{0}).
	\end{align*}
Consequently, since $u(x_{0})>0$, it follows that
\begin{equation}
\label{5.2}
	u^{q-1}(x_{0})-\lambda u(x_{0})-1\leq 0.
\end{equation}
Equivalently, $P(u(x_0))\leq 0$, where $P(\l,z):=z^{q-1}-\lambda z-1$, $z\geq 0$. Since ${P}(\l,0)=-1<0$ and $\lim_{z\to\infty} {P}(\lambda,z)=+\infty$, the function ${P}(\l,\cdot)$ possesses a positive zero for every $\l\in\mathbb{R}$. Let $C(\lambda)>0$ be the minimal positive zero of ${P}(\lambda,\cdot)$ for each $\lambda\in\R$. According to \eqref{5.2}, $u(x_{0})\leq C(\lambda)$.
\par
Finally, setting $r(\lambda):=\lambda^{\frac{1}{q-2}}+1$, we have that,  for every $\l\geq 0$,
\begin{align*}
	{P}(\lambda,r(\lambda)) &= (\l^\frac{1}{q-2}+1)^{q-1} -\l (\l^\frac{1}{q-2}+1) -1
  =  \sum_{i=0}^{q-1} \left(\!\!\! \begin{array}{c} q-1\\ i \end{array}\!\!\! \right)
 \l^\frac{i}{q-2}-\l^\frac{q-1}{q-2} -\l-1 \\ & =
 \left(\!\!\! \begin{array}{c} q-1\\ q-2 \end{array}\!\!\! \right)\l +
 \sum_{i=0}^{q-3} \left(\!\!\! \begin{array}{c} q-1\\ i \end{array}\!\!\! \right)
 \l^\frac{i}{q-2} -\l  =
	(q-2)\lambda+\sum_{i=1}^{q-3}\binom{q-1}{i}\lambda^{\frac{i}{q-2}}\geq 0
\end{align*}
and hence $C(\lambda)\leq r(\lambda)$. This ends the proof if $\l\geq 0$. When $\l<0$, we have that
$$
  {P}(\l,1)=-\l>0,\qquad \frac{d P}{d z}(\l,z)=(q-1)z^{q-2}-\l >0,
$$
for all $z\geq 0$ and $\l<0$. Thus, $C(\l)\leq 1$ if $\l<0$. This ends the proof.
\end{proof}

The next result provides us with a necessary condition for the existence of positive solutions when $d\geq \s_{1}^{-1}$.

\begin{lemma}
\label{le5.2}
Suppose $d\geq\s_{1}^{-1}$ and \eqref{i.1} admits a positive solution. Then,  $\lambda >0$.
\end{lemma}
\begin{proof}
Let $(\lambda,u)\in\mathscr{S}_{d}$. Then, multiplying the differential equation by $u$ and integrating yields
\begin{equation*}
	\int_{\Omega}(-d\Delta u) u \, dx=\lambda \sum_{i=1}^{N} a_{i} \int_{\Omega} \frac{\partial u}{\partial x_{i}} u \ dx+\int_{\Omega}u^{2} \, dx +\lambda \int_{\Omega}u^{3} \, dx-\int_{\Omega}u^{q+1} \, dx.
\end{equation*}
Thus, integrating by parts in $\O$, it follows from the Courant characterization of $\s_1$ that
\begin{equation*}
	d\s_{1}\int_{\Omega}u^{2} \, dx \leq d\int_{\Omega}|\nabla u|^{2} \, dx=\int_{\Omega}u^{2} \, dx+\int_{\Omega}(\lambda-u^{q-2})u^{3} \, dx.
\end{equation*}
Consequently, since $1\leq d\s_{1}$, it is apparent that
\begin{equation*}
	0\leq (d\s_{1}-1)\int_{\Omega}u^{2} \ dx\leq \int_{\Omega}(\lambda-u^{q-2})u^{3} \ dx,
\end{equation*}
and therefore, $\lambda> 0$. This concludes the proof.
\end{proof}

The following result shows that $\mc{P}_\l (\mathscr{S}_d)$ is bounded, where $\mc{P}_\l$ stands for the $\l$-projection operator defined by $\mc{P}_\l(\l,u)=\l$ for all
$\l\in\R$ and $u\in W^{2,p}_0(\O)$. Part (ii) straightens Lemma \ref{le5.2}.

\begin{lemma}
\label{le5.3}
The following assertions are true:
	\begin{enumerate}
		\item[{\rm (i)}]  Suppose $d\leq \s_{1}^{-1}$. Then, there exists a constant $C_0(d)>0$ such that \begin{equation}
\label{v.3}
		-\tfrac{2}{|\mathfrak{a}|}\sqrt{d(1-\s_{1}d)}\leq\lambda\leq  C_0(d)\quad
\hbox{for all}\;\; (\l,u)\in\mathscr{S}_d.
\end{equation}
		\item[{\rm (ii)}] Suppose $d > \s_{1}^{-1}$. Then, there are two constants,
$0<C_1(d)<C_2(d)$, such that
\begin{equation}
\label{v.4}
		C_1(d) \leq \lambda \leq  C_2(d) \quad \hbox{for all}\;\; (\l,u)\in\mathscr{S}_d.
\end{equation}
Moreover, $C_{1}(d), C_{2}(d)\to \infty$ as $d\to\infty$.
\end{enumerate}
\end{lemma}

\begin{proof}
Let $(\lambda,u)\in\mathscr{S}_{d}$. Then, $v(x):=e^{\frac{\lambda}{2d}\langle \mathfrak{a},x\rangle}u(x)$, $x\in\bar \Omega$, satisfies $v=0$ on $\p\O$ and
\begin{equation}
\label{v.5}
	-\Delta v=\tfrac{1}{d}\left(1-\tfrac{\lambda^{2} |\mathfrak{a}|^{2}}{4d}\right)v+f_{d}(\lambda,x,v)v
\quad \hbox{in}\;\; \Omega,
\end{equation}
where $f_{d}$ is defined as in the proof of Proposition \ref{pr4.4}.
Since $v>0$, by the uniqueness of the principal eigenvalue it is apparent that
\begin{equation}
	\label{v.6}
	\sigma_{1}[-\Delta- f_{d}(\lambda,x,v),\Omega]= \tfrac{1}{d}\left(1-\tfrac{\lambda^{2} |\mathfrak{a}|^{2}}{4d}\right).
\end{equation}
Suppose that $d\leq \s_{1}^{-1}$ and $\l\geq 0$. Then, by Lemma \ref{le5.1},
\begin{equation*}
	f_{d}(\lambda,x,v)\leq d^{-1}\lambda\zeta(\lambda,d,x)v=d^{-1}\lambda u\leq d^{-1}\l C(\l).
\end{equation*}
Thus, by the monotonicity of the principal eigenvalue with respect to the potential,
\begin{equation*}
	\tfrac{1}{d}\left(1-\tfrac{\lambda^{2} |\mathfrak{a}|^{2}}{4d}\right)\geq
      \s_{1}-d^{-1}\lambda C(\lambda).
\end{equation*}
Therefore, rearranging terms yields
\begin{equation}
	\label{v.7}
	 \Gamma(\l):= \tfrac{\l^{2}|\mathfrak{a}|^{2}}{4d^2}-d^{-1}\l(1+\l^{\frac{1}{q-2}})\leq d^{-1}-\s_{1}.
\end{equation}
Let $\Lambda_{d}$ be the set of $\l\geq 0$ satisfying \eqref{v.7}. Clearly $\Lambda_{d}$ is a closed connected subset of $[0,\infty)$. As $\Gamma(\l)<0$ for sufficiently small $\l>0$, there exists $\varepsilon>0$ such that $[0,\varepsilon)\subset \Lambda_{d}$. On the other hand, $\lim_{\l\uparrow \infty}\Gamma(\l)=+\infty$. This shows Part (i) when $\l\geq 0$.
Now, suppose  $\l<0$. Then, $f_{d}(\lambda,x,v)\leq 0$ and, hence, \eqref{v.6} implies that
	\begin{equation*}
	\tfrac{1}{d}\left(1-\tfrac{\lambda^{2} |\mathfrak{a}|^{2}}{4d}\right)\geq \s_{1},
	\end{equation*}
which provides us with the lower estimate of \eqref{v.3} and ends the proof Part (i).
\par
Now, suppose $d>\s_{1}^{-1}$. According to Lemma \ref{le5.2}, \eqref{i.1} cannot admit a positive solution
if $\l \leq 0$. Thus, we can assume $\l>0$. By adapting the argument of the proof of Part (i),
it becomes apparent that $\Gamma(\l)\leq d^{-1}-\s_{1}<0$ if $(\l,u)\in\mathscr{S}_d$. As above, let $\Lambda_d$ denote the set of $\l>0$ satisfying $\Gamma(\l)\leq d^{-1}-\s_{1}$. If $\Lambda_{d}=\emptyset$, then the conclusion follows by choosing $C_{1}(d)=C_{2}(d)=d>0$. Suppose $\Lambda_{d}$ is non-empty.
In such case, since $\Gamma(\l)\uparrow 0$ as $\l\downarrow 0$, $(0,\varepsilon)\cap\Lambda_{d}=\emptyset$ for some $\varepsilon>0$. Thus, $\min \Lambda_{d}>0$ and we can choose $C_{1}(d)=\min \Lambda_{d}$. On the other hand, as $\lim_{\l\ua\infty}\Gamma(\l)=+\infty$ and $d^{-1}-\s_{1}<0$, it follows that $\max\Lambda_{d}<\infty$ and we choose $C_{2}(d)=\max\Lambda_{d}$. This shows \eqref{v.4}. Since $C_{1}(d), C_{2}(d)\in \Lambda_{d}$, it follows that
\begin{equation*}
0<\sigma_{1}d-1\leq C_{i}(d)[1+C_{i}(d)]^{\frac{1}{q-2}}, \quad i\in\{1,2\}.
\end{equation*}
This shows $C_{1}(d), C_{2}(d)\to \infty$ as $d\to\infty$. The proof is complete.
\end{proof}

Subsequently,  for any given compact interval $J$ of $\R$, we denote by $\mathscr{S}_{d}(J)$ the set of positive solutions $(\l,u)\in\mathscr{S}_d$ with $\l\in J$. The next result provides us with uniform a priori bounds for these subsets of $\mathbb{R}\times W^{2,p}_{0}(\Omega)$.

\begin{theorem}
\label{th5.4}
For any compact interval  $J\subset \mathbb{R}$, there is a constant $C(J,d)>0$ such that
\begin{equation}
\label{v.8}
	\sup_{(\l,u)\in \mathscr{S}_{d}(J)}\|u\|_{W^{2,p}}\leq C(J,d).
\end{equation}
\end{theorem}
\begin{proof}
Let $J$ be a compact interval of $\R$ and pick  $(\lambda,u)\in\mathscr{S}_{d}(J)$. We already know that
the function $v(x):=e^{\frac{\lambda}{2d}\langle \mathfrak{a},x\rangle}u(x)$, $x\in\O$, satisfies $v=0$ on $\p\O$ and \eqref{v.5}. By the invertibility of the operator $-\D :W^{2,p}_0(\O)\to L^p(\O)$, based on a
classical inequality of Calder\'{o}n and Zygmund \cite{CZ} (see \cite[Ch. 9]{GT}, if necessary), there exists a constant $M>0$, depending only on $p$ and $N$ (the spatial dimension), such that
\begin{equation}
\label{v.9}
	\|v\|_{W^{2,p}}\leq M \left( \tfrac{1}{d}\left|1-\tfrac{\lambda^{2} |\mathfrak{a}|^{2}}{4d}\right|\|v\|_{L^{p}}+\|f_{d}(\l,x,v)v\|_{L^{p}}\right).
\end{equation}
Since $J$ and $\bar \Omega$ are compact, the function $e^{\frac{\lambda}{2d}\langle \mathfrak{a},x\rangle}$ is uniformly bounded on $(\l,x)\in J\times\overline{\Omega}$ for each $d>0$. Subsequently, we set
\begin{equation*}
	E(d,J):=\max_{(\lambda,x)\in J\times \overline{\Omega}}e^{\frac{\lambda}{2d}\langle \mathfrak{a},x\rangle}, \qquad     L(J):=\max_{\l\in J}|\l|.
\end{equation*}
Then, thanks to Lemma \ref{le5.1}, we have that
\begin{equation}
\label{v.10}
	\|v\|_{L^{\infty}}\leq E(d,J)\|u\|_\infty \leq E(d,J)(|\l|^{\frac{1}{q-2}}+1)\leq E(d,J)(L(J)^{\frac{1}{q-2}}+1),
\end{equation}
Thus, by the definition of $f_d$,  there exists a constant $F(d,J)>0$ such that
\begin{equation}
\label{v.11}
	\|f_{d}(\lambda,x,v)v\|_\infty \leq F(d,J).
\end{equation}
Combining \eqref{v.9} with \eqref{v.10} and \eqref{v.11}, the estimate \eqref{v.8} readily follows.
\end{proof}

\section{A priori bounds for the negative Solutions}

\noindent To get a priori bounds for the negative solutions is a more delicate issue, as it depends
on whether the exponent $q\geq 4$ is odd or even. Actually, the global structure of the set of negative solutions of \eqref{i.1} changes, very substantially, in these two cases, as it will become apparent later. As a  consequence, these two cases will be treated separately.

\subsection{$q\geq 4$ is an odd integer} For any given $d>0$, we will denote by  $\mathscr{N}_{d}$ the set of negative solutions of \eqref{i.1}, i.e.,
\begin{equation*}
\mathscr{N}_{d}:=\{(\l,u)\in\mathfrak{F}_{d}^{-1}(0): \;u\ll 0\}\subset \mathbb{R}\times W^{2,p}_{0}(\Omega).
\end{equation*}
Setting $v=-u$, it becomes apparent that the negative solutions of \eqref{i.1} are given by the positive solutions of
\begin{equation}
\label{vi.1}
\left\{ \begin{array}{ll} -d\Delta v=\lambda \langle \mathfrak{a},\nabla v\rangle+v-\lambda v^{2}-v^{q}
 & \quad \hbox{in}\;\; \Omega,\\ v=0& \quad \hbox{on}\;\; \partial\Omega.\end{array}\right.
\end{equation}
Based on the fact that, much like in \eqref{i.1},  the dominant term at $v=+\infty$ is $-v^q$, the set of positive solutions of \eqref{vi.1} satisfies similar properties as
the set of positive solutions of \eqref{i.1} already analysed in Section 5. As the proofs can be easily adapted, to avoid repetitions we will restrict ourselves to state the corresponding results without proofs. The next results provide us with counterparts of Lemmas \ref{le5.1}, \ref{le5.2}, \ref{le5.3} and Theorem
\ref{th5.4}, respectively.

\begin{lemma}
\label{le6.1}
There exists a real valued function $C:\mathbb{R}\to(0,+\infty)$ such that, for every
$d>0$ and  $(\lambda,u)\in\mathscr{N}_{d}$,
\begin{equation}
\label{v.2}
	\|u\|_{\infty}\leq C(\l)\leq \left\{\begin{array}{ll} 1-\lambda^{\frac{1}{q-2}} & \quad \hbox{if}\;\; \lambda\leq 0, \\ 1 & \quad \hbox{if}\;\; \lambda > 0. \end{array}\right.
\end{equation}
\end{lemma}

\begin{lemma}
\label{le6.2}
$\lambda <0$ if $d\geq\s_{1}^{-1}$ and \eqref{i.1} admits a negative solution.
\end{lemma}

\begin{lemma}
\label{le6.3}
The following assertions are true:
	\begin{enumerate}
		\item[{\rm (i)}]  Suppose $d\leq \s_{1}^{-1}$. Then, there exists a constant $C_0(d)>0$ such that \begin{equation}
\label{vi.3}
		-C_0(d) \leq\lambda\leq  \tfrac{2}{|\mathfrak{a}|}\sqrt{d(1-\s_{1}d)} \quad
\hbox{for all}\;\; (\l,u)\in\mathscr{N}_d.
\end{equation}
		\item[{\rm (ii)}] Suppose $d > \s_{1}^{-1}$. Then, there are two constants,
$C_1(d)<C_2(d)<0$, such that
\begin{equation}
\label{vi.4}
		C_1(d) \leq \lambda \leq  C_2(d) \quad \hbox{for all}\;\; (\l,u)\in\mathscr{N}_d.
\end{equation}
\end{enumerate}
Moreover, $C_{1}(d), C_{2}(d)\to -\infty$ as $d\to\infty$. In particular, \eqref{i.1} cannot admit a negative solution if $d > \s_{1}^{-1}$ and $\l\geq 0$.
\end{lemma}

\begin{theorem}
\label{th6.4}
For any compact interval  $J\subset \mathbb{R}$, there is a constant $C(J,d)>0$ such that
\begin{equation}
\label{vi.5}
	\sup_{(\l,u)\in \mathscr{N}_{d}(J)}\|u\|_{W^{2,p}}\leq C(J,d).
\end{equation}
\end{theorem}

\subsection{$q\geq 4$ is an even integer} In such case, the negative solutions of \eqref{i.1} are given through the change of variable $v=-u$ from the positive solutions of
\begin{equation}
\label{vi.6}
\left\{ \begin{array}{ll} -d\Delta v=\lambda \langle \mathfrak{a},\nabla v\rangle+v-\lambda v^{2}+v^{q}
 & \quad \hbox{in}\;\; \Omega,\\ v=0& \quad \hbox{on}\;\; \partial\Omega,\end{array}\right.
\end{equation}
which is a much more sophisticated problem than Problem \eqref{vi.1}, as it is of superlinear type with  dominant term at $v=+\infty$ given by $v^q$. Thus, when $N\geq 3$, the existence of a priori bounds relays on the size of $q$ with respect to the critical exponent $\frac{N+2}{N-2}$, much like in the classical papers of Gidas and Spruck \cite{GS}, \cite{GS2}, whose finding  were adapted to study a general class of superlinear
indefinite problems by Berestycki, Capuzzo-Dolcetta and Nirenberg \cite{BCN} and Amann and L\'{o}pez-G\'{o}mez \cite{ALG}. Yet \eqref{vi.6} lies outside the general scope of these papers.
\par
Note that the change of variable
$$
    v(x)=\zeta(\lambda,d,x) w(x), \quad \zeta(\lambda,d,x) = e^{-\frac{\lambda}{2d}\langle \mathfrak{a},x\rangle} ,\quad x\in\Omega,
$$
transforms the problem \eqref{vi.6} into
\begin{equation}
\label{vi.7}
\left\{ \begin{array}{ll} -\Delta w=\tfrac{1}{d}\left(1-\tfrac{\lambda^{2}|\mathfrak{a}|^{2}}{4d}\right)w-\lambda d^{-1}\zeta w^{2}+d^{-1}\zeta^{q-1} w^{q} & \quad \hbox{in} \;\; \Omega,\\ w=0 & \quad \hbox{on}\;\;\p\O,
\end{array}\right.
\end{equation}
where $\zeta\equiv \zeta(\l,d,x)$. Since, the differential equation of \eqref{vi.7} cannot be expressed in the form  $-\Delta w=\mu w+a(x)w^{q}$ for some continuous function $a(x)$, our next results are not a direct consequence of the findings of \cite{GS}, \cite{GS2}, \cite{BCN} and \cite{ALG}. However, the blowing-up techniques introduced by Gidas and Spruck \cite{GS}, \cite{GS2} can be adapted to get them, like in \cite{BCN} and \cite{ALG}.
\par
As in Section 6.1, for any given compact subinterval $J\subset\R$, $\mathscr{N}_{d}(J)$ stands for the
set of negative solutions of \eqref{i.1}.

\begin{theorem}
\label{th6.5} Suppose $q\geq 4$ is an even integer and either $N=1, 2$, or $N\geq 3$ and $ q<\frac{N+2}{N-2}$. Then, for every compact interval $J\subset \mathbb{R}$, there exists a constant $C(J,d)$ such that
\begin{equation*}
	\sup_{(\l,u)\in \mathscr{N}_{d}(J)}\|u\|_{W^{2,p}}<C(J,d).
\end{equation*}
\end{theorem}
Theorem \ref{th6.5} provides us with a priori bounds for the negative solutions of \eqref{i.1} only when $N=1, 2$, or $N=3$ and $q=4$, though it is an optimal result because it is well known that the a priori bounds are lost when $N\geq 3$ and $q\geq \frac{N+2}{N-2}$.

\begin{proof}
Fix a compact interval $J\subset \mathbb{R}$. By elliptic regularity, arguing as in the proof of
Theorem \ref{th5.4}, it suffices to show the existence of a constant $C(J,d)$ such that
\begin{equation}
\label{vi.8}
	\sup_{(\l,u)\in \mathscr{N}_{d}(J)}\|u\|_\infty <C(J,d).
\end{equation}
On the contrary, suppose that \eqref{vi.8} fails. Then, there exist a sequence $\{(\lambda_{n},w_{n})\}_{n\in\mathbb{N}}$ of positive solutions of \eqref{vi.7} in $J\times W^{2,p}_{0}(\Omega)$
and a sequence $\{x_{n}\}_{n\in\mathbb{N}}$ in $\Omega$ such that
\begin{equation}
\label{vi.9}
	M_{n}:=\|w_{n}\|_\infty =w_{n}(x_{n}), \qquad \lim_{n\to\infty} M_{n} = +\infty.
\end{equation}
Since $\bar \Omega$ is compact, there exists a subsequence of $\{x_{n}\}_{n\in\mathbb{N}}$, still
labeled by $n$, such that
\begin{equation}
\label{vi.10}
	\lim_{n\to\infty} x_{n}=x_\infty \in\bar \Omega.
\end{equation}
Since $\bar\O=\O\cup\p\O$, either $x_\infty \in\Omega$, or $x_\infty \in\partial\Omega$.
\par
\vspace{0.2cm}
\noindent \textbf{Case 1: Suppose that $x_\infty \in\Omega$.} Then, for every $n\geq 1$, we consider the re-scaled function
\begin{equation}
\label{vi.11}
	\nu_{n}:=M_{n}^{2-q},\qquad \tilde{w}_{n}\left(\nu_{n}^{\frac{1-q}{2(q-2)}}(x-x_{n})\right):=\nu_{n}^{\frac{1}{q-2}}w_{n}(x), \qquad x\in\Omega,
\end{equation}
which differs from the classical one of Gidas and Spruck \cite{GS}, \cite{GS2}.
By \eqref{vi.11}, regardless the domain of definition of $\tilde w$, we have that, for every $n\geq 1$,
\begin{equation}
\label{vi.12}
   \|\tilde w_n\|_\infty = \tilde w_n(0)=\nu_{n}^\frac{1}{q-2} w_n(x_n)= \nu_{n}^\frac{1}{q-2}M_n =1
\end{equation}
Moreover, by \eqref{vi.9} and \eqref{vi.11}, $\lim_{n\to\infty}\nu_n=0$. To estimate the domain
of definition of $\tilde w_n$, pick any $\e>0$ satisfying
$$
   0<\varepsilon<\min_{n\in\mathbb{N}}\mathrm{dist}(x_{n},\partial\Omega)
$$
and let $n_0\in\mathbb{N}$ be an integer such that
\begin{equation}
\label{vi.13}
	|\mathrm{dist}(x_\infty,\partial\Omega)-\mathrm{dist}(x_{n},\partial\Omega)|<\varepsilon \quad \hbox{for all}\;\; n\geq n_0.
\end{equation}
Note that, setting, $s:=\frac{q-1}{2(q-2)}$, one has that
$$
   \tilde{w}_{n}\left(\nu_{n}^{-s}(x-x_{n})\right):=\nu_{n}^{\frac{1}{q-2}}w_{n}(x), \qquad x\in\Omega.
$$
Thus,  for every $n\geq 1$, the domain of definition of
$\tilde w_n$ is the set
$$
  \mathcal{D}(\tilde w_n) :=\nu_n^{-s}( -x_n+\bar \O),
$$
because $\O$ is the domain of definition of $w_n$. We claim that, setting
$$
  \r_n := \tfrac{\mathrm{dist}(x_\infty,\p\O)}{2}\nu_n^{-s},\qquad n\geq 1,
$$
one has that $B_{\r_n}\equiv B_{\r_n}(0)\subset \mc{D}(\tilde w_n)$ for all $n\geq n_0$, where
$B_R$ is the ball of radius $R>0$ centered at $0$.  Indeed, by choice of $\e$, it follows from \eqref{vi.13} that, for every $n\geq n_0$,
\begin{equation*}
	\tfrac{\mathrm{dist}(x_\infty,\p\O)}{2} < \tfrac{\varepsilon+\mathrm{dist}(x_n,\p\O)}{2}<  \mathrm{dist}(x_n,\p\O).
\end{equation*}
Thus, for every $y\in B_{\r_n}$
with $n\geq n_0$, we have that
$$
  \|y\|\leq \r_n=\tfrac{\mathrm{dist}(x_\infty,\p\O)}{2}\nu_n^{-s}<  \mathrm{dist}(x_n,\p\O)\nu_n^{-s}.
$$
Hence $\|\nu_n^s y\|\leq \mathrm{dist}(x_n,\p\O)$, and so
$x_n+\nu_n^s y \in\O$, for all $n\geq n_0$. Therefore,
\begin{equation}
\label{vi.14}
    B_{\r_n}\subset\mc{D} (\tilde w_n),\qquad n\geq n_0.
\end{equation}
Moreover, since $\lim_{n\to\infty}\nu_n=0$, one has that $\lim_{n\to \infty}\r_n=+\infty$ and hence,
for every $R>0$ one can enlarge $n_0$, if necessary, so that  $R<\r_n$, and hence $B_R\subset \mc{D}(\tilde w_n)$, for all $n\geq n_0$. By differentiating and substituting in \eqref{vi.7}, it becomes apparent that, for any given $R>0$ and $n\geq n_0$, $\tilde{w}_{n}\in W^{2,p}(B_R)$ and it solves
\begin{equation}
\label{vi.15}
	-d\Delta \tilde{w}_{n}=\nu_{n}^{\frac{q-1}{q-2}}\left(1-\tfrac{\lambda_{n}^{2}|\mathfrak{a}|^{2}}{4d} \right)\tilde{w}_{n}-\nu_{n}\lambda_{n} \zeta_{n}\tilde{w}_{n}^{2}+\zeta_n^{q-1}\tilde{w}_{n}^{q}
\end{equation}
in $B_R$, point-wise almost everywhere,  where
$$
  \zeta_n:= \zeta(\l_n,d,x_n+\nu_{n}^{s}y).
$$
According to \eqref{vi.12}, we have that
$$
  \|\tilde w_n\|_{\mc{C}(\bar B_R)}=\tilde w_n(0)=1\quad \hbox{for all}\;\; n\geq n_0.
$$
Therefore, using the $L^p$-theory as in the proof of Theorem \ref{th5.4}, there exists a constant $C=C(J,d)$ such that
\begin{equation*}
	\|\tilde{w}_{n}\|_{W^{2,p}}\leq C(J,d)\quad \hbox{for all}\;\; n\geq n_0,
\end{equation*}
where we are denoting by $C(J,d)$ any constant depending on $d$ and $J$. Consequently, the sequence
$\{\tilde{w}_{n}\}_{n\geq n_0}$ is uniformly bounded in $W^{2,p}({B}_{R})$. Thus, by the compactness of the imbedding $W^{2,p}({B}_{R})\hookrightarrow \mc{C}^{1,1-\frac{N}{p}}(\bar B_R)$, we can extract a
subsequence, $\{\tilde{w}_{n_{k}}\}_{k\in\mathbb{N}}$, such that, for some $w\in W^{2,p}({B}_{R})$,
\begin{equation*}
	\lim_{k\to\infty}\tilde{w}_{n_{k}}=w
\end{equation*}
weakly in $W^{2,p}({B}_{R})$ and strongly in $W^{1,p}({B}_{R})$ and in
$\mathcal{C}^{1,1-\frac{N}{p}}(\bar B_{R})$. Since $\tilde{w}_{n_{k}}(0)=1$ for all $k\in\mathbb{N}$,
it follows that $w(0)=1$. As  $J$ is compact, without loss of generality, we can assume that, for some $\l_\infty\in J$,  $\lim_{k\to\infty} \lambda_{n_k} = \lambda_\infty$. Thus, letting $k\to \infty$ in \eqref{vi.15} at $n=n_k$, each side converges weakly in $L^{p}(B_{R})$ and strongly in $\mc{C}^1(\bar B_R)$ to
\begin{equation*}
	-d\Delta w=\zeta^{q-1}(\l_\infty,d,x_\infty)w^{q}.
\end{equation*}
As $R>0$ is arbitrary, through a further diagonal argument, we can assume that $w$ is actually defined in
the whole of $\R^N$. Since $w \in \mc{C}^1(\R^N)$, by elliptic regularity,
$w \in \mc{C}^2(\R^N)$. Moreover, by construction,  $w\geq 0$, $w(0)=1$ and $\|w\|_\infty\leq 1$.
Consequently, since $\zeta^{q-1}(\l_\infty,d,x_\infty)>0$, this contradicts \cite[Th. 1.1]{GS2}, because we are assuming that either $N=1, 2$, or $N=3$ and $q=4$. So, $q<\frac{N+2}{N-2}$.
\par
\vspace{0.2cm}
\noindent \textbf{Case 2: Suppose that $x_\infty \in\partial\Omega$.} By a change of variable depending only on $\O$, we can assume  that $x_\infty=0$ and there exists a neighborhood of $x_\infty=0$
in $\R^N$, $\mathscr{U}$,  such that
\begin{equation}
\label{vi.16}
\begin{split}
  \mathscr{U}\cap\p\O & =\{x=(x_1,...,x_N)\in\mathscr{U}:\;\; x_N=0\},\\
  \mathscr{U}\cap \O & =\{x=(x_1,...,x_N)\in\mathscr{U}:\;\; x_N>0\}.
\end{split}
\end{equation}
As in the proof of Case 1, for every $n\geq 1$, we consider the re-scaled function defined through
\eqref{vi.11}. Similarly, setting
$$
   s:=\frac{q-1}{2(q-2)}, \qquad \d_n:= \mathrm{dist}(x_{n},\partial\Omega)=x_{n,N},\qquad
   r_n:=\nu_{n}^{-s}\d_n, \qquad n\geq 1,
$$
the domain of definition of $\tilde w_n$, $\mc{D}(\tilde w_n)$, consists of the set of points $y\in\R^N$ such that $x=x_n+\nu_n^s y\in \bar \O$. In particular, it contains the set of pints $y\in \R^N$ such that
$x_n+\nu_n^s y\in \bar \O\cap \mathscr{U}$. Hence, it follows that the condition $\|y\|\leq r_{n}$ together with
\begin{equation}
\label{vi.17}
    x_{n,N}+\nu_n^s y_N \geq 0
\end{equation}
entails $x_n+\nu_n^s y\in \bar \O\cap \mathscr{U}$ and hence, $y\in \mc{D}(\tilde w_n)$. As \eqref{vi.17} can be equivalently expressed as
$$
   y_N \geq -\nu_n^{-s} x_{n,N}=-\nu_n^{-s}\d_n=-r_n,
$$
it becomes apparent that
\begin{equation}
\label{vi.18}
  \mathscr{D}(\tilde w_n) :=\{y\in B_{r_{n}}:y_{N}\geq -r_n\}\subset \mc{D}(\tilde w_n)
  \qquad \hbox{for all}\;\; n\geq 1.
\end{equation}
Therefore, $\tilde w_n$ is well defined in $\mathscr{D}(\tilde w_n)$ for all $n\geq 1$.
\par
As in Case, 1, we have that
\begin{equation}
\label{vi.19}
	\lim_{n\to \infty} \nu_{n}=0,\qquad \|\tilde{w}_{n}(y)\|_{\mathcal{C}(\mathscr{D}(\tilde w_n))}=\tilde{w}_{n}(0)=\nu_{n}^{\frac{1}{q-2}}M_{n}=1.
\end{equation}
Similarly, $\tilde{w}_{n}\in W^{2,p}(\mathscr{D}(\tilde w_n))$ and it satisfies
\eqref{vi.15} in  $\mathscr{D}(\tilde w_n)$.
\par
By elliptic regularity, thanks to \eqref{vi.15} and \eqref{vi.19}, there exists a
constant $C>0$ such that
$$
  \| \nabla \tilde{w}_{n}\|_{\mc{C}(\mathscr{D}(\tilde w_n))}\leq C \quad \hbox{for all}\;\; n\geq 1.
$$
Thus, the mean value theorem implies that, for every $n\geq 1$,
\begin{equation}
\label{vi.20}
	|\tilde{w}_{n}(0)-\tilde{w}_{n}(0,\cdots,0,-r_n)|\leq \|\nabla
    \tilde{w}_{n}\|_{\mc{C}(\mathscr{D}(\tilde w_n))} r_n \leq C r_n.
\end{equation}
On the other hand, by \eqref{vi.11} and \eqref{vi.16}, we find that
\begin{equation*}
\tilde{w}_{n}(0,\cdots,0,-r_n)=
\nu_{n}^{\frac{1}{q-2}}w_{n}(x_{n,1},\cdots,x_{n,N},0)=0,
\end{equation*}
and, thanks to \eqref{vi.19}, $\tilde w_n(0)=1$. Therefore, substituting in \eqref{vi.20} yields
$1 \leq C r_n$ for all $n\geq 1$.  In other words, the sequence $\{r_n\}_{n\in\mathbb{N}}$ is separated away
from zero.
\par
There are two possibilities: Either $\lim_{n\to\infty}r_n=+\infty$,  or there exists a subsequence,
labeled again by $n$, such that $\lim_{n\to \infty}r_n =r$ for some $r>0$.
\par
Suppose $\lim_{n\to\infty}r_n=+\infty$. Then, for every $R>0$, there exists $n_1\in\N$ such that
$R<r_n$ for all $n\geq n_1$. In this case, since $\mathscr{D}(\tilde w_n)$ approximates $\R^N$ as $n\to \infty$, adapting the argument of the last part of the proof of Case 1,
we can again reach a contradiction with Theorem 1.1 of Gidas and Spruck \cite{GS2}.
\par
Suppose that $\lim_{n\to \infty}r_n =r$ for some $r>0$. Then, setting
$$
   \mathbb{H}:=\{y\in\R^N:\;\; y_N > -r\},
$$
and adapting the proof of the Case 1, we get the existence of a function $w\in \mc{C}^2(\mathbb{H})$ such that $0\leq w \leq 1$, $w(0)=1$, $w=0$ on $\p \mathbb{H}$, and
\begin{equation*}
	-d\Delta w=\zeta^{q-1}(\l_\infty,d,x_\infty)w^{q} \quad \hbox{in}\;\;\mathbb{H}.
\end{equation*}
Since $q<\frac{N+2}{N-2}$ and $\zeta^{q-1}(\l_\infty,d,x_\infty)>0$, this contradicts \cite[Th. 1.3]{GS} and
 ends the proof.
\end{proof}

We end this section with a result that will be useful later.

\begin{lemma}
\label{le6.6}
The problem \eqref{i.1} cannot admit a negative solution if  $d\leq \s_{1}^{-1}$ and $\l=0$.
\end{lemma}
\begin{proof}
As the negative solutions of \eqref{i.1} are given by the positive solutions of \eqref{vi.6} via the change of variables $v=-u$, it suffices to show that the problem
\begin{equation}
\label{vi.21}
\left\{ \begin{array}{ll} -d\Delta v=v+v^{q}
 & \quad \hbox{in}\;\; \Omega,\\ v=0& \quad \hbox{on}\;\; \partial\Omega,\end{array}\right.
\end{equation}
cannot admit a positive solution if $d\leq \s_{1}^{-1}$. Let $v$ be a positive solution of \eqref{vi.21}.
Then, multiplying the $v$-equation by a principal eigenfunction, $\v_0$, associated to  $\s_1$, integrating by parts in $\O$ and rearranging terms yields
\begin{equation*}
	(d\s_{1}-1)\int_{\O}v\varphi_{0} \, dx=\int_{\Omega}v^{q}\varphi_{0} \, dx.
\end{equation*}
Since $\int_{\Omega}v^{q}\varphi_{0} \, dx>0$, this cannot occur if $d\leq\s_{1}^{-1}$.
\end{proof}

\section{Global bifurcation diagrams}

\noindent In this section, we ascertain the global structure of the set of positive and negative solutions of \eqref{i.1}. Recall that the solutions of \eqref{i.1} are the zeroes of the nonlinear differential operator
$\mathfrak{F}_{d}:\mathbb{R}\times W^{2,p}_{0}(\Omega)\to L^{p}(\Omega)$ defined by
\begin{equation}
\label{7.1}
\mathfrak{F}_{d}(\lambda,u)=d\Delta u+\lambda\langle \mathfrak{a}, \nabla u\rangle + u+\lambda u^{2}-u^{q}.
\end{equation}
The next result establishes that the positive (resp. negative) solutions of  \eqref{i.1} cannot leave the interior of the positive cone of the ordered Banach space $\mathcal{C}^{1}_0(\bar \Omega)$ unless they reach $u=0$.

\begin{lemma}
\label{le7.1}
	Let $\{(\lambda_{n},u_{n})\}_{n\in\mathbb{N}}$ be a sequence of positive (resp. negative) solutions of \eqref{i.1} such that
	\begin{equation}
	\label{E7.2}
	\lim_{n\to\infty}(\lambda_{n},u_{n})=(\lambda_{0},u_{0})\in\mathfrak{F}_{d}^{-1}(0), \quad \text{ in } \mathbb{R}\times W^{2,p}_{0}(\Omega).
	\end{equation}
	Then, either $u_{0}\gg 0$ (resp. $u_{0}\ll 0$), or $u_{0}=0$.
\end{lemma}

\begin{proof}
	We will prove it for the case of positive solutions. By equation \eqref{E7.2} and the Sobolev embedding $W^{2,p}(\O)\hookrightarrow \mc{C}^{1,1-\frac{N}{p}}(\bar\O)$, it follows that  $u_{0}\in\mathfrak{F}_{d}^{-1}(0)$ is the pointwise limit of positive functions, $u_{n}\gg 0$, $n\in\mathbb{N}$. Hence $u_{0}=0$, concluding the proof, or $u_{0}\gneq 0$. In the later case, Theorem \ref{th2.4} is applied to obtain $u_0\gg 0$. The proof is complete.
\end{proof}

 The next result establishes a pivotal compactness property of  $\mf{F}_{d}$.

\begin{lemma}
\label{le7.2}
For every $d>0$, $\mf{F}_{d}$ is proper on closed and bounded subsets of $\mathbb{R}\times W^{2,p}_{0}(\O)$. \end{lemma}

\begin{proof}
It suffices to prove that the restriction of $\mf{F}_{d}$ to the closed subset
${K}:=[\lambda_{-},\lambda_{+}]\times \bar {B}_{R}$ is proper, where $\lambda_{-}<\lambda_{+}$ and $B_{R}$ stands for the open ball of $W^{2,p}_{0}(\O)$ of radius $R>0$ centered at $0$. According to
\cite[Th. 2.7.1]{B}, we must check that $\mf{F}_{d}({K})$ is closed in $L^{p}(\O)$, and that, for every $f\in L^{p}(\O)$, the set $\mf{F}_{d}^{-1}(f)\cap {K}$ is compact in $\mathbb{R}\times W^{2,p}_{0}(\O)$.
\par
To show that $\mf{F}_{d}(K)$ is closed in $L^{p}(\O)$, let $\{f_{n}\}_{n\in\mathbb{N}}$ be a sequence in $\mf{F}_{d}(K)\subset L^{p}(\O)$ such that
\begin{equation}
\label{7.2}
	\lim_{n\to\infty} f_{n}=f \quad \text{in } L^{p}(\O).
\end{equation}
Then, there exists a sequence $\{(\l_{n},u_{n})\}_{n\in\mathbb{N}}$ in $K$ such that \begin{equation}
\label{7.3}
	f_{n}=\mf{F}_{d}(\l_{n},u_{n}) \quad \hbox{for all} \;\; n\in\mathbb{N}.
\end{equation}
By the compactness of the imbedding $W^{2,p}(\O)\hookrightarrow \mc{C}^{1,1-\frac{N}{p}}(\bar\O)$, we can extract a subsequence $\{(\l_{n_{k}},u_{n_{k}})\}_{k\in\mathbb{N}}$ such that, for some $(\l_{0},u_{0})\in [\l_-,\l_+]\times \mc{C}^{1,1-\frac{N}{p}}(\bar\O)$,   $\lim_{k\to \infty} \l_{n_{k}} = \l_{0}$ and
\begin{equation}
	\label{7.4}
	\lim_{k\to\infty} u_{n_{k}} =u_{0} \quad \text{in } \mc{C}^{1,1-\frac{N}{p}}(\bar\O).
\end{equation}
As a direct consequence of \eqref{7.2}, \eqref{7.3} and \eqref{7.4}, it becomes apparent that
$u_{0}$ must be a weak solution of the nonlinear elliptic problem
\begin{equation}
	\label{7.5}
\left\{ \begin{array}{ll} d\Delta u_{0}+\l_{0} \langle \mf{a},\nabla u_{0}\rangle +
    u_0+\lambda u_0^{2}-u_0^{q} =f & \quad \hbox{in}\;\; \O,
    \\ u_0 = 0 & \quad \hbox{on}\;\;\p\O. \end{array}\right.
\end{equation}
By elliptic regularity, $u_{0}\in W^{2,p}_{0}(\O)$ and $f=\mf{F}_{d}(\l_{0},u_{0})$. Therefore, $f\in \mf{F}_{d}(K)$.
\par
Now, pick $f\in L^{p}(\O)$. To show that $\mf{F}_{d}^{-1}(f)\cap K$ is compact in $[\l_-,\l_+]\times W^{2,p}_{0}(\O)$. Let $\{(\l_{n},u_{n})\}_{n\in\mathbb{N}}$ be a sequence in $\mf{F}_{d}^{-1}(f)\cap K$. Then,
\begin{equation}
\label{7.6}
	\mf{F}_{d}(\l_{n},u_{n})=f \quad \hbox{for all}\;\; n\in\mathbb{N}.
\end{equation}
Based again on the compactness of the imbedding $W^{2,p}(\O)\hookrightarrow \mc{C}^{1,1-\frac{N}{p}}(\bar\O)$, we can extract a subsequence $\{(\l_{n_{k}},u_{n_{k}})\}_{k\in\mathbb{N}}$ such that, for some $(\l_{0},u_{0})\in [\l_-,\l_+]\times \mc{C}^{1,1-\frac{N}{p}}(\bar\O)$,   $\lim_{k\to \infty} \l_{n_{k}} = \l_{0}$ and \eqref{7.4} holds.
Similarly,  $u_{0}\in \mc{C}^{1,1-\frac{N}{p}}(\bar\O)$ is a weak solution of \eqref{7.5} and, by elliptic regularity,
$u_{0}\in W^{2,p}_{0}(\O)$ and $\mf{F}_{d}(\l_{0},u_{0})=f$. In particular,
\begin{equation*}
	-d\Delta (u_{n_{k}}-u_{0})=\l_{0}\langle \mf{a}, \nabla (u_{n_{k}}-u_{0})\rangle +\lambda_{0} (u^{2}_{n_{k}}-u^{2}_{0})-(u^{q}_{n_{k}}-u^{q}_{0}) \quad \hbox{in}\;\; \O, \qquad k\in\mathbb{N}.
	\end{equation*}
By the $L^p$-elliptic estimates, there is a positive constant $C>0$ such that
\begin{equation*}
	\|u_{n_{k}}-u_{0}\|_{W_0^{2,p}(\O)}\leq C \left( \|u^{2}_{n_{k}}-u^{2}_{0}\|_{L^{p}(\O)}+ \|u^{q}_{n_{k}}-u^{q}_{0}\|_{L^{p}(\O)}\right)
\end{equation*}
for all $k\in\N$. On the other hand, $\{u_{n_k}\}_{k\in\N}$ is bounded in $W^{2,p}(\O)$ and hence it is relatively compact in $\mc{C}^{1,1-\frac{N}{p}}(\bar\O)$. Therefore, letting $k\to \infty$ we finally get
that
$$
   \lim_{k\to \infty} (\l_{n_{k}},u_{n_{k}}) = (\l_{0},u_{0})\quad \hbox{in}\;\;
   [\l_-,\l_+] \times W^{2,p}_{0}(\O).
$$
This concludes the proof.
\end{proof}

The rest of this section is devoted to the analysis of the global structure of the set of positive and negative solutions of \eqref{i.1}. As it is strongly dependent on the size of the diffusion coefficient $d>0$, we have divided it into three subsections.
\vspace{0.2cm}

\subsection{The case when $d<\s_{1}^{-1}$} Throughout this section we assume that $d\s_1<1$. To get our main results in this case, we will use the unilateral global bifurcation theorem \cite[Th. 8.5]{JJ4}, which is a refinement of \cite[Th. 6.4.3]{LG01} and \cite[Th. 1.2]{LGRI}. To state it, consider two real Banach spaces, $U, V$,  and an operator $\mathfrak{F}\in\mathcal{C}^1(\mathbb{R}\times U,V)$ satisfying:
\begin{enumerate}
\item[(C)] $U$  is a subspace of $V$ with compact inclusion  $U \hookrightarrow V$.
\item[(F1)] $\mathfrak{F}(\lambda,0)=0$ for all $\lambda\in\mathbb{R}$.
\item[(F2)] $D_{u}\mathfrak{F}(\lambda,u)\in\Phi_{0}(U,V)$ for all $\lambda\in\mathbb{R}$ and $u\in U$.
\item[(F3)] $\mf{F}$ is proper on closed and bounded subsets of $\R\times U$.
\item[(F4)] The map
\begin{equation}
\label{7.7}
	\mf{N}(\l,u):= \mf{F}(\l,u)-D_u\mf{F}(\l,0)u,\qquad (\l,u)\in \R\times U,
\end{equation}
admits a continuous extension, also denoted by $\mf{N}$, to $\R \times V$.
\item[(F5)] The linealization $\mf{L}(\l):=D_u\mf{F}(\l,0)$ is analytic in $\l\in\mathbb{R}$ and $\l_0$ is an isolated eigenvalue of $\mf{L}(\l):=D_u\mf{F}(\l,0)$ such that
$N[\mathfrak{L}(\lambda_{0})]=\mathrm{span}[\varphi_{0}]$ for some $\varphi_0\in U$ with  $\|\v_0\|=1$.
\end{enumerate}
We consider for $\varepsilon>0$ and $\eta\in(0,1)$, the open subsets of $\mathbb{R}\times U$,
\begin{equation*}
Q^{+}_{\varepsilon,\eta}:=\{(\l,u)\in\mathbb{R}\times U: |\l-\l_{0}|<\varepsilon, \ \langle \varphi^{\ast}_{0},u\rangle>\eta \|u\| \},
\end{equation*}\begin{equation*}
Q^{-}_{\varepsilon,\eta}:=\{(\l,u)\in\mathbb{R}\times U: |\l-\l_{0}|<\varepsilon, \ \langle \varphi^{\ast}_{0},u\rangle<-\eta \|u\| \},
\end{equation*}
We denote by $\mathcal{T}$ the set of trivial solutions of $\mf{F}(\l,u)=0$, that is, $\mathcal{T}:=\{(\l,0):\l\in\mathbb{R}\}$. The set $\mf{F}^{-1}(0)\backslash\mc{T}$ is consequently referred as the set of non-trivial solutions. Then, under these hypotheses, the following result holds.

\begin{theorem}
\label{th7.3}
	Suppose {\rm (C)}, $\mf{F}$ satisfies {\rm (F1)--(F5)} and
	$\chi[\mathfrak{L},\l_{0}]\in 2\mathbb{N}+1$. Then, there exist two connected components of $\mathfrak{F}^{-1}(0)\backslash\mc{T}$, denoted by $\mathfrak{C}^{+}$ and $\mf{C}^{-}$, such that $(\l_{0},0)\in\overline{\mf{C}^{\pm}}$ and for sufficiently small $\delta>0$,
	$$\mathfrak{C}^{+}\cap B_{\delta}(\lambda_{0},0)\subset Q^{+}_{\varepsilon,\eta}, \quad \mathfrak{C}^{-}\cap B_{\delta}(\lambda_{0},0)\subset Q^{-}_{\varepsilon,\eta},$$
	for every $\varepsilon\in (0,\varepsilon_{0})$ and $\eta\in(0,1)$, for some $\varepsilon_{0}>0$. On the other hand, let $Z\subset U$ a closed subspace such that
$$
	U= N[\mf{L}_0]\oplus Z,\qquad N[\mf{L}_0]=\mathrm{span}[\v_0].
$$
	Then, each unilateral component $\mf{C}^\nu$, $\nu\in\{\pm\}$, satisfies some of the following alternatives:
	\begin{enumerate}
		\item[{\rm (i)}] $\mf{C}^\nu$ is not compact in $\R\times U$.
		\item[{\rm (ii)}] There exists $\mu \neq \l_0$ such that $(\mu,0)\in \overline{\mf{C}^\nu}$.
		\item[{\rm (iii)}]  There exist $\l\in\R$ and $z\in Z\setminus\{0\}$ such that $(\l,z)\in \mf{C}^\nu$.
	\end{enumerate}
\end{theorem}

This theorem will provide us with the global behavior of the continua of positive and negative solutions of \eqref{i.1} when $d<\s_{1}^{-1}$. By the local analysis of Sections 2 and 3, we already know that from each of the points $(\l_{1}(d),0)$ and $(-\l_{1}(d),0)$ there emanates an analytic curve of positive solutions of \eqref{i.1}. Let us denote by $\mathscr{C}^{+}_{+}$ and $\mathscr{C}^{+}_{-}$ the connected components of the set of non-trivial solutions $\mf{F}_{d}^{-1}(0)\backslash\mathcal{T}$ containing the curves of positive solutions emanating from $(\l_{1}(d),0)$ and $(-\l_{1}(d),0)$, respectively. By Lemma \ref{le7.1} they can only leave the interior of the positive cone through $u=0$ and hence they consists of positive solutions, that is, $\mathscr{C}^{+}_{\pm}\subset\mathscr{S}_{d}$. The next result, based on
Theorem \ref{th7.3}, shows that $\mathscr{C}^{+}_{+}=\mathscr{C}^{+}_{-}$, as illustrated by Figure \ref{F1}.
\begin{theorem}
\label{th7.4}
It holds that $\mathscr{C}^{+}_{+}=\mathscr{C}^{+}_{-}$. Thus,
there is a connected component of the set of positive solutions $\mathscr{C}^{+}(=\mathscr{C}^{+}_{+}=\mathscr{C}^{+}_{-})\subset \mathscr{S}_{d}$ such that  $(\pm\l_{1}(d),0)\in\overline{\mathscr{C}^{+}}$.
\end{theorem}

\begin{proof}
We will apply  Theorem \ref{th7.3} to the operator $\mathfrak{F}_{d}:\mathbb{R}\times W^{2,p}_{0}(\O)\to L^{p}(\O)$ defined by \eqref{7.1}. Clearly, $\mathfrak{F}_{d}$ satisfies the hypotheses of  Theorem \ref{th7.3} with  $\lambda_{0}=\pm\lambda_{1}(d)$. Thus, the connected components $\mathscr{C}^{+}_{\pm}\subset\mf{F}_{d}^{-1}(0)\backslash\mathcal{T}$
are well defined and each of them satisfies one of the alternatives (i)--(iii). According to Lemma \ref{le5.3}(i), there exist $\a<\b$ such that $\l\in[\a,\b]$ if \eqref{i.1} admits a positive solution $(\l,u)$. Thus, by Theorem \ref{th5.4}, the alternative (i) cannot occur. To exclude (iii) we take
\begin{equation*}
	Z=\Big\{u\in W^{2,p}_{0}(\O): \int_{\O}u\varphi_{0} \ dx=0\Big\},
\end{equation*}
where $\varphi_{0}\gg 0$ is any principal eigenfunction associated $\s_1$. Suppose that there are $\l\in\mathbb{R}$ and $v\in Z\backslash\{0\}$ such that $(\l,v)\in\mathscr{C}^{+}_{+}$. Since $v\neq 0$,
necessarily $v\gg 0$ and hence  $\int_{\O}v\varphi_{0} \, dx >0 $ which contradicts $v\in Z$.
Hence, the alternative (iii) cannot occur neither. Therefore, there exists $\mu\neq \l_{1}(d)$ such that $(\mu,0)\in\overline{\mathscr{C}^{+}_{+}}$. In particular, this implies that $\l=\mu$ is a bifurcation value to positive solutions from $u=0$. Thus, by Theorem \ref{th3.2}(iv),  $\mu =-\l_{1}(d)$. Therefore,
$\mathscr{C}_{+}^{+}=\mathscr{C}_{-}^{+}$.
\end{proof}

As far as concerns the negative solutions, by the local analysis of Sections 2 and 3, from each of the points $(\l_{1}(d),0)$ and $(-\l_{1}(d),0)$  emanates an analytic curve of negative solutions. Let denote by  $\mathscr{C}^{-}_{+}$ and $\mathscr{C}^{-}_{-}$ the connected components of the set of non-trivial solutions $\mf{F}_{d}^{-1}(0)\backslash\mathcal{T}$ that contains the curves of negative solutions emanating from $(\l_{1}(d),0)$ and $(-\l_{1}(d),0)$, respectively. By Lemma \ref{le7.1} they can only leave the interior of the negative cone through $u=0$ and hence they consists of negative solutions, that is, $\mathscr{C}^{-}_{\pm}\subset\mathscr{N}_{d}$. By the results of Section 6, these components
might have a different behavior according to the oddity of $q$. The next result provides us with their behavior when $q\geq 4$ is odd. It has been sketched in the left plot of Figure \ref{F1}.

\begin{theorem}
\label{th7.5}
Let $q\geq 4$ be an odd integer. Then, $\mathscr{C}^{-}_{+}=\mathscr{C}^{-}_{-}$, i.e., there exists a connected component of the set of negative solutions $\mathscr{C}^{-}(=\mathscr{C}^{-}_{+}=\mathscr{C}^{-}_{-})\subset \mathscr{N}_{d}$ linking  $(\l_{1}(d),0)$ to $(-\l_{1}(d),0)$.
\end{theorem}
\begin{proof}
It follows identical patterns as the proof of Theorem \ref{th7.4}, though now one should use
Lemma \ref{le6.3}(i), instead of Lemma \ref{le5.3}(i). So, we omit the technical details.
\end{proof}

The behavior of $\mathscr{C}^{-}_{+}$ and $\mathscr{C}^{-}_{-}$ is rather different when
$q\geq 4$ is even. Actually, the next theorem establishes that, in this case, they are disjoint and unbounded.

\begin{theorem}
\label{th7.6}
Suppose $q\geq 4$ is an even integer. Then, $\mathscr{C}^{-}_{+}$ and $\mathscr{C}^{-}_{-}$ are
unbounded and disjoint, i.e., $\mathscr{C}^{-}_{+}\cap\mathscr{C}^{-}_{-}=\emptyset$. Moreover, if
 $N=1, 2$, or $N\geq 3$ and $q=4$, then
\begin{equation}
\label{7.8}
	(\lambda_{1}(d),\infty)\subset\mathcal{P}_{\l}(\mathscr{C}^{-}_{+}), \quad (-\lambda_{1}(d),-\infty)\subset\mathcal{P}_{\l}(\mathscr{C}^{-}_{-}),
\end{equation}
where $\mathcal{P}_{\l}$ stands for the $\l$-projection operator. In particular, \eqref{i.1}
has at least one negative solution whenever $|\l|>\l_{1}(d)$, as illustrated by the right plot
of Figure \ref{F1}.
\end{theorem}

\begin{proof}
By Lemma \ref{le6.6}, \eqref{i.1} cannot admit a negative solution at $\l=0$. Thus,
$\mathscr{C}^{-}_{+}\cap\mathscr{C}^{-}_{-}=\emptyset$. According to Theorem \ref{th7.3},
$\mathscr{C}^{-}_{+}$ and $\mathscr{C}^{-}_{-}$ must satisfy some of the alternatives (i)--(iii).
By Theorem \ref{th3.2}(iv) and Lemma \ref{le6.6}, the alternative (ii) cannot occur. Arguing as in
the proof of Theorem \ref{th7.5}, the option (iii) is excluded to occur too. Therefore,  (i) occurs, i.e.,
$\mathscr{C}^{-}_{+}$ and $\mathscr{C}^{-}_{-}$ are unbounded. This concludes the proof of the first part of the result. Finally, suppose that $N=1, 2$, or $N=3$ and $q=4$. Then, thanks to Theorem \ref{th6.5}, for
every compact interval $J\subset\mathbb{R}$, the subsets
	\begin{equation*}
	\mathscr{C}^{-}_{\pm}(J):=\{(\l,u)\in\mathscr{C}^{-}_{\pm}: \;\, \l\in J\}\subset\mathscr{C}^{-}_{\pm}
	\end{equation*}
	are bounded in $\mathbb{R}\times W^{2,p}_{0}(\O)$. As $\mathscr{C}^{-}_{\pm}$ are unbounded,
\eqref{7.8} holds.
\end{proof}
\vspace{0.2cm}
\par

\subsection{The case when $d=\s_{1}^{-1}$} This section shows the validity of the global bifurcation diagrams sketched in Figure \ref{F2}, by invoking \cite[Th. 6.5]{JJ4}, which follows by adapting some arguments of Dancer \cite{Da,Da73,Da732} and Buffoni and Tolland \cite{BT}. For any
proper analytic map, $\mathfrak{F}:\mathbb{R}\times U\to V$, satisfying (F1)--(F3) of Section 4.2,
$(\lambda,u)\in\mathbb{R}\times U$ is said to be \emph{regular} with respect to $\mathfrak{F}$ if $D_{u}\mathfrak{F}(\lambda,u)\in GL(U,V)$. In our setting, \cite[Th. 6.5]{JJ4} reads as follows.

\begin{theorem}
\label{th7.7}
Let $\mathfrak{F}\in\mathcal{H}(\mathbb{R}\times U,V)$ be an analytic map satisfying {\rm (F1)--(F3)} of Section 4.2 such that it is is proper on bounded and closed subsets of $\mathbb{R}\times U$. Suppose that
	$\mf{F}^{-1}(0)$ possesses a local analytic branch, $\g:(0,\e)\to \mathbb{R}\times U$, emanating from $(0,0)$ and consisting of regular points for sufficiently small $\e>0$. Then, $\g$ admits  a prolongation to a global locally injective continuous path $\G: (0,+\infty)\to \mathbb{R}\times U$ on $\mathfrak{F}^{-1}(0)$ satisfying one of the following non-excluding alternatives. Either
	\begin{enumerate}
		\item[{\rm (a)}]   $\lim_{t\ua \infty} \|\Gamma(t)\|_{\mathbb{R}\times U}= +\infty$, or
		\item[{\rm (b)}]   $\Gamma$ is a closed loop, i.e., there exists $T>0$   such that $\Gamma(T)=(0,0)$.
	\end{enumerate}
\end{theorem}

First of all, we will prove that, for any integer  $q\geq 4$, there is a loop of positive solutions of \eqref{i.1} emanating from $u=0$ at $\l=\pm\l_1(\s_1^{-1})=0$. The existence of a connected component of the set of positive solutions $\mathscr{S}_{d}$, bifurcating from $(\lambda,u)=(0,0)$ has been already established in Section 4. More precisely,
we already know that there emanate from $(0,0)$  two analytic arcs of positive solutions $\gamma_{i}:(0,\varepsilon)\to \mathbb{R}\times W^{2,p}_{0}(\Omega)$, $\gamma_{i}(\lambda)=(\lambda,u_{i}(\lambda))$, with  $\lim_{\lambda\downarrow 0}u_{i}(\lambda)=0$, $i\in\{1,2\}$. The connected components of the set of positive solutions $\mathscr{S}_{d}$ containing to each of the curves $\gamma_{1}$ and $\gamma_{2}$, locally at $(0,0)$, will be called $\mathscr{C}^{+}_{1}$ and $\mathscr{C}^{+}_{2}$, respectively.

\begin{theorem}
\label{th7.8}
Under the previous assumptions, $\mathscr{C}^{+}_{1}=\mathscr{C}^{+}_{2}$. Moreover each of the local curves $\gamma_{i}:(0,\varepsilon)\to \mathbb{R}\times W^{2,p}_{0}(\Omega)$ can be continued to a global locally injective continuous curve $\Gamma_{i}: (0,T)\to\mathscr{C}^{+}_{i}$ such that $\Gamma_{i}|_{[T-\delta,T)}=\gamma_{j}$ for some $\delta>0$ and  $j\in\{1,2\}\backslash\{i\}$. Thus, there is a loop of positive solutions of \eqref{i.1} with vertex at $(0,0)$.
\end{theorem}

\begin{proof}
Once given the local curve $\gamma_{1}:(0,\varepsilon)\to \mathbb{R}\times W^{2,p}_{0}(\Omega)$ and the component  $\mathscr{C}^{+}_{1}$, in order to apply Theorem \ref{th7.7}, we should make sure that, for sufficiently small $\varepsilon>0$, the set $\gamma_{1}((0,\varepsilon))\subset \mathfrak{F}_{d}^{-1}(0)$
consists of regular points of $\mathfrak{F}_{d}$. By the local analysis already done in Section 4, the regular and singular points of $\mathfrak{F}_{d}$ in $\mathfrak{F}_{d}^{-1}(0)\cap \mathcal{V}$ are in analytical correspondence with those of the reduced map $\mathfrak{G}_{d}(\l,x)=xg_{d}(\l,x)$ in  $\mathfrak{G}_{d}^{-1}(0)\cap \mc{U}$, where $\mathcal{V}$ and $\mathcal{U}$ are open neighborhoods of $\mathbb{R}\times U$ and $\mathbb{R}^{2}$, respectively, containing $(0,0)$. So, it suffices to prove that, near $(0,0)$, the set $\mathfrak{G}_{d}^{-1}(0)$ does not contain any singular point of $\mathfrak{G}_{d}$ different form $(0,0)$. By Theorem \ref{th3.2}(ii), $\chi[\mathfrak{L}_{d},0]=2$. Thus,
 it follows from \eqref{nueva} that
\begin{equation*}
	\chi[\mathfrak{L}_{d},0] = \ord_{\l=0}g_{d}(\l,0)=2.
\end{equation*}
Consequently, by the Weierstrass--Malgrange preparation theorem, shortening the neighborhood $\mathcal{U}=\mathcal{U}_{\lambda}\times\mathcal{U}_{x}\subset\mathbb{R}^{2}$,  if necessary,  there exists an analytic function $c:\mathcal{U}\to\mathbb{R}$ such that  $c(0,0)\neq 0$, plus $\chi=2$ analytic functions, $c_{j}:\mathcal{U}_{x}\to\mathbb{R}$, $c_{j}(0)=0$, $j=1,2$, such that
	\begin{equation*}
	g_{d}(\lambda,x)=c(\lambda,x)\left[\lambda^{2}+c_{1}(x)\lambda+c_{2}(x)\right].
	\end{equation*}
Hence, we can rewrite $\mathfrak{G}_{d}:\mathcal{U}\to\mathbb{R}$ as
$$
  \mathfrak{G}_{d}(\lambda,x)=xc(\lambda,x)\left[\lambda^{2}+c_{1}(x)\lambda+c_{2}(x)\right].
$$
By the local analysis already done in Section 4 (see Figure \ref{F}), for every $x\in \mathcal{U}_{x}\backslash\{0\}$, the equation $\mathfrak{G}_{d}(\l,x)=0$ has two positive different solutions in $\lambda\in\mathcal{U}_{\lambda}$. Thus, there are two analytic maps,  $\varphi_{j}:(-\delta,\delta)\backslash\{0\}\to\mathbb{R}$, $j=1,2$, such that
	\begin{equation*}
	\mathfrak{G}_{d}(\l,x)=x c(\l,x)(\lambda-\varphi_{1}(x))(\l-\varphi_{2}(x)), \quad x\in\mathcal{U}_{x}\backslash\{0\}.
	\end{equation*}
By a direct computation if follows that $(\l,x)\in\mathfrak{G}_{d}^{-1}(0)\cap\mathcal{U}$, $(\l,x)\neq(0,0)$, is a singular point, i.e.,  $D_{x}\mathfrak{G}_{d}(\lambda,x)=0$,  if and only if $\varphi_{1}(x)=\varphi_{2}(x)$ or $\varphi'_{j}(x)=0$ for some $j=1,2$. According to
\eqref{iv.7} and \eqref{iv.8}, for sufficiently small $\mathcal{U}$, this is not possible. Therefore, $\gamma_{1}:(0,\varepsilon)\to\mathbb{R}\times W^{2,p}_{0}(\Omega)$ consists of regular points for
sufficiently small $\varepsilon>0$. By Theorem \ref{th7.7}, $\gamma_{1}$ admits a prolongation
to a global locally injective continuous map $\Gamma_{1}:(0,\infty)\to\mathbb{R}\times W^{2,p}_{0}(\Omega)$ on $\mathfrak{F}_{d}^{-1}(0)$ satisfying one of the alternatives (a) or (b). Due to Lemma \ref{le7.1}, $\Gamma_{1}(0,\infty)\subset\mathscr{C}^{+}_{1}$. Thanks to Theorem \ref{th5.4}, $\Gamma_{1}(0,\infty)$ is bounded. Therefore, the alternative (a) cannot occur. Consequently, there exists some $T>0$ such that $\Gamma_{1}(T)=(0,0)$. As in a neighborhood of $(0,0)$ the set of positive solutions consists of
the graphs of $\gamma_{1}$ and $\gamma_{2}$, being  $\Gamma_{1}$ is locally injective, it follows that, modulus a re-parametrization (if necessary), $\Gamma_{1}|_{(T-\delta,T]}=\gamma_{2}$. This implies, in particular, that $\mathscr{C}^{+}_{1}=\mathscr{C}^{+}_{2}$ and concludes the proof.
\end{proof}

As illustrated by Figure \ref{F2}, the behavior of the negative solutions differs according to the oddity of $q$.
\par
Suppose $q\geq 5$ is odd.  Then, by the local analysis of Section 4 summarized in Figure \ref{F}, we already know that in a neighborhood of $(\lambda,u)=(0,0)$ there emanate two analytic arcs of negative solutions,  $\gamma_{i}:(-\varepsilon,0)\to \mathbb{R}\times W^{2,p}_{0}(\Omega)$, $\gamma_{i}(\lambda)=(\lambda,u_{i}(\lambda))$, such that $\lim_{\lambda\uparrow 0}u_{i}(\lambda)=0$, $i\in\{1,2\}$. The components of the set of negative solutions $\mathscr{N}_{d}$ containing the curves $\gamma_{1}$ and $\gamma_{2}$ will be subsequently denoted by $\mathscr{C}^{-}_{1}$ and $\mathscr{C}^{-}_{2}$, respectively.
In this case, adapting the proof of Theorem \ref{th7.8}, the following result holds.

\begin{theorem}
\label{th7.9}
$\mathscr{C}^{-}_{1}=\mathscr{C}^{-}_{2}$ if $q\geq 5$ is odd. Moreover each of the local curves $\gamma_{i}:(-\varepsilon,0)\to \mathbb{R}\times W^{2,p}_{0}(\Omega)$  can be continued to a global locally injective continuous curve $\Gamma_{i}: (-T,0)\to\mathscr{C}^{-}_{i}$ such that $\Gamma_{i}|_{(-T,-T+\delta]}=\gamma_{j}$ for some $\d>0$ and $j\in\{1,2\}\backslash\{i\}$.
Thus, there is a loop of negative solutions of \eqref{i.1} with vertex at $(0,0)$, as sketched by the left picture of Figure \ref{F2}.
\end{theorem}

Now, suppose that $q$ is even. Then, according to the analysis carried out in Section 4, we already know that there are two analytic curves of negative solutions bifurcating from $(0,0)$: One in the direction of $\l>0$ and another in the direction of $\l<0$. Subsequently, we denote by $\mathscr{C}^{-}_{+}$ (resp. $\mathscr{C}^{-}_{-}$)  the connected component of the set of negative solutions $\mathscr{N}_{d}$ emanating from $(0,0)$ in the direction $\l>0$ (resp. $\l<0$). The next result provides us with their global behavior.

\begin{theorem}
	Suppose  $q\geq 4$ is an even integer. Then, $\mathscr{C}^{-}_{+}$ and  $\mathscr{C}^{-}_{-}$ are unbounded and disjoint, i.e., $\mathscr{C}^{-}_{+}\cap\mathscr{C}^{-}_{-}=\emptyset$. Moreover, if
 $N=1, 2$, or $N\geq 3$ and $q=4$, then
	\begin{equation}
\label{7.9}
	\mathcal{P}_{\l}(\mathscr{C}^{-}_{+})=(0,\infty), \quad \mathcal{P}_{\l}(\mathscr{C}^{-}_{-})=(-\infty,0),
	\end{equation}
as sketched in the right picture of Figure \ref{F2}.
In particular, \eqref{i.1} possesses at least one negative solution for every  $\l\neq 0$.
\end{theorem}

\begin{proof}
By Lemma \ref{le6.6}, \eqref{i.1} cannot admit a negative solution at $\l=0$. So,
$\mathscr{C}^{-}_{+}\cap\mathscr{C}^{-}_{-}=\emptyset$. Let us denote by $\gamma_{+}:(0,\varepsilon)\to\mathbb{R}\times W^{2,p}_{0}(\Omega)$ and $\gamma_{-}:(-\varepsilon,0)\to\mathbb{R}\times W^{2,p}_{0}(\Omega)$ the two local curves of negative solutions of \eqref{i.1} that emanate from $(0,0)$ in the direction of $\mathscr{C}^{-}_{+}$ and $\mathscr{C}^{-}_{-}$, respectively. Adapting the argument of the proof of Theorem \ref{th7.8}, it
is easily seen  that $\gamma_{+}$ and $\gamma_{-}$ consist of regular points for sufficiently small $\varepsilon>0$. Thus, by Theorem \ref{th7.7}, there are two global locally injective continuous curves $\Gamma_{+}:(0,\infty)\to\mathbb{R}\times W^{2,p}_{0}(\Omega)$ and $\Gamma_{-}:(-\infty,0)\to\mathbb{R}\times W^{2,p}_{0}(\Omega)$ that extend $\gamma_{+}$ and $\gamma_{-}$, respectively, and satisfy one of the alternatives (a) and (b). By Lemma \ref{le7.1},  $\Gamma_{+}((0,\infty))\subset\mathscr{C}^{-}_{+}$ and $\Gamma_{-}((-\infty,0))\subset \mathscr{C}^{-}_{-}$. Since $\mathscr{C}^{-}_{+}\cap\mathscr{C}^{-}_{-}=\emptyset$, the curves $\Gamma_{\pm}$ cannot form a loop.
Thus, the  alternative (a) cannot happen. Therefore,
$$
    \lim_{t\uparrow\infty}\|\Gamma_{+}(t)\|_{W^{2,p}}=\infty, \quad \lim_{t\downarrow-\infty}\|\Gamma_{-}(t)\|_{W^{2,p}}=\infty.
$$
Owing to Theorem \ref{th6.5}, this entails that $\mathscr{C}^{-}_{+}$ and $\mathscr{C}^{-}_{-}$ are unbounded. Consequently,  $\mathcal{P}_{\l}(\mathscr{C}^{-}_{+})=(0,\infty)$ and $\mathcal{P}_{\l}(\mathscr{C}^{-}_{-})=(-\infty,0)$. This ends the proof.
\end{proof}

The components $\mathscr{C}^{-}_{+}$ and $\mathscr{C}^{-}_{-}$ might loose their a priori bounds at some critical values of $\l$, $\l^*_\pm$, if $N\geq 3$ and $q\geq 6$.

\subsection{The case when $d>\s_{1}^{-1}$}
In such case, due to Theorem \ref{th3.2}(iii),   $\Sigma(\mathfrak{L}_{d})=\emptyset$. Thus, neither the positive solutions nor the negative solutions can bifurcate from $u=0$. Thus, to get the existence of positive solutions, we  proceeded through an indirect argument involving the analytic implicit function theorem (see, e.g.,  \cite[Th. 3.3.2]{B}). As a result, for $d$ sufficiently close to  $\s_{1}^{-1}$, \eqref{i.1} admits, at least, one compact connected component of the set of positive solutions $\mathscr{C}^{+}\subset \mathscr{S}_{d}$ separated away from $u=0$, as illustrated in Figure \ref{F3}. Precisely, the following result holds.

\begin{theorem}
\label{th7.11}
There exists $\nu>0$ such that, for every $d\in (\s_{1}^{-1},\nu)$, the problem \eqref{i.1} has a
compact connected component of the set of positive solutions, $\mathscr{C}^{+}_{d}\subset\mathscr{S}_{d}$, such that $\mathcal{P}_{\l}(\mathscr{C}^{+}_{d})= [\a,\b]$ for some $0<\a\leq \b$.
\end{theorem}
\begin{proof}
By our previous results in  Section 7.2, we already know that there emanates two branches of analytical curves of positive solutions of $\mathfrak{F}_{d}(\l,u)=0$ from $(0,0)$ at $d=\s_{1}^{-1}$. Moreover, these curves are filled in by regular points. Therefore, there exists $(\l_{0},\s_{1}^{-1},u_{0})\in\mathfrak{F}^{-1}(0)$, with $u_{0}\gg 0$ and $\l_{0}>0$,  such that
$$
  D_{u}\mathfrak{F}(\l_{0},\s_{1}^{-1},u_{0})\in GL(W^{2,p}_{0}(\O),L^{p}(\Omega)).
$$
By the analytic implicit fuction theorem, there exists an open neighborhood of $(\l_{0},\s_{1}^{-1})$, $\mathcal{U}\subset\mathbb{R}\times\mathbb{R}_{+}$, and an analytic mapping, $U:\mathcal{U}\to W^{2,p}_{0}(\Omega)$, such that
\begin{equation*}
    	U(\l_{0},\s_{1}^{-1})=u_{0}, \quad \mathfrak{F}(\l,d,U(\l,d))=0, \ (\l,d)\in\mathcal{U}.
\end{equation*}
By Lemma \ref{le7.1}, $U(\mathcal{U})$ consists of positive solutions for sufficiently small $\mc{U}$.
Consequently, there exists $\nu>0$ such that, for every $d\in(\s_{1}^{-1},\nu)$, there is some $\l=\l_d>0$  for which the problem \eqref{i.1} admits a positive solution, $u_d$. Let $\mathscr{C}^{+}_{d}\subset\mathscr{S}_{d}$
be the connected component of the set of positive solutions through $(\l_d,u_d)$. According to Lemma \ref{le5.3}(ii),
there exist $0<C_1(d)<C_2(d)$ such that $C_1(d) \leq \lambda \leq  C_2(d)$ for all $(\l,u)\in\mathscr{C}^+_d$.
Moreover, thanks to Theorem \ref{th5.4}, there exists a constant $C>0$ such that
\begin{equation*}
    	\sup_{(\l,u)\in\mathscr{C}^{+}_d}\|u\|_{W^{2,p}}<C.
\end{equation*}
Therefore, since by Lemma \ref{le7.2}, the operator $\mf{F}_{d}$ is proper on closed and bounded subsets, $\mathscr{C}^+_d$ is a compact connected component of the set of positive solutions. As it is separated away from $u=0$, because $\Sigma(\mathfrak{L}_{d})=\emptyset$, it becomes apparent that
$\mathcal{P}_{\l}(\mathscr{C}^{+}_d)= [\a,\b]$ for some $\a\leq \b$, because
$\mathscr{C}^+_d$ is compact and connected and $\mathcal{P}_\l$ is continuous. This ends the proof.
\end{proof}

As far as concerns the negative solutions, as usual, their structure depends on the values of $q$.
Indeed, when $q$ is an odd integer, adapting the proof of Theorem \ref{th7.11}, it is easily seen that the following result holds.

\begin{theorem}
\label{th7.12}
Suppose $q\geq 4$ is an even integer. Then, there exists $\nu>0$ such that, for every $d\in (\s_{1}^{-1},\nu)$, the problem \eqref{i.1} has a
compact connected component of negative solutions, $\mathscr{C}^{-}_{d}\subset \mathscr{N}_{d}$, such that $\mathcal{P}_{\l}(\mathscr{C}^{-}_d)= [\a,\b]$ for some $\a\leq \b<0$, as illustrated in the left plot of Figure \ref{F3}.
\end{theorem}

Finally, suppose that $q$ is an odd integer. Then, although the argument of the proof of Theorem \ref{th7.11} provides us with a connected component, $\mathscr{C}^-_d$, of the set of negative solutions of \eqref{i.1} separated away from
$u=0$ for every $d>\s_1^{-1}$ sufficiently close to $\s_1^{-1}$, and $\mathscr{C}_d^-$ possesses uniform a priori bounds on compact intervals of $\l$ if $N=1, 2$, or $N=3$ and $q=4$, we do not know yet whether, or not, $\mathscr{C}_d^-$ is bounded, or semi-bounded, or  simply $\mc{P}_\l(\mathscr{C}_d^-)=\R$, as suggested by the right plot of Figure \ref{F3}. This remains an open problem in this paper.


\begin{thebibliography}{xx}


\bibitem{ALG} H. Amann and J. L\'{o}pez-G\'{o}mez, A  priori bounds and
multiple solutions for superlinear indefinite elliptic problems,
\emph{J. Differential Equ.} \textbf{146} (1998), 336--374.

\bibitem{BCN} H. Berestycki, I. Capuzzo-Dolcetta and L. Nirenberg, Superlinear indefinite elliptic problems and nonlinear Liouville theorems, \emph{Top. Meth. Nonl. Anal.}\textbf{4} (1994), 59--78.
		
\bibitem{B} M. S. Berger, \emph{Nonlinearity and Functional Analysis}, Lectures on Nonlinear Problems in Mathematical Analysis, Academic Press, Inc., 1977.

\bibitem{Bo} J. M. Bony,  Principe du maximum dans les espaces de Sobolev, \emph{C.
R. Acad. Sci. Paris} \textbf{265} (1967), 333--336.
		
		\bibitem{BT} B. Buffoni and J. Toland, \emph{Analytic Theory and Global Bifurcation: An introduction}, Princeton Series in Applied Mathematics, Princeton, 2003.

\bibitem{CZ} A. P. Calder\'{o}n and A. Zygmund, On the existence of certain singular
integrals, \emph{Acta Math.} \textbf{88} (1952), 85--139.

\bibitem{Cano} S. Cano-Casanova, Compact components of positive solutions for superlinear indefinite elliptic problems of mixed typem, \emph{Topol. Methods Nonlinear Anal.} \textbf{23} (2004), 45--72.

\bibitem{CLGM} S. Cano-Casanova, J. L\'{o}pez-G\'{o}mez and M. Molina-Meyer, Isolas: compact solution components separated away from a given equilibrium curve, \emph{Hiroshima Math. J.} \textbf{34}
    (2004), 177--199.
		
		\bibitem{CR} M. G. Crandall and P. H. Rabinowitz, Bifurcation from simple eigenvalues,
		\emph{J. Funct. Anal.} \textbf{8} (1971), 321--340.
		
		\bibitem{CRex}  M. G. Crandall and P. H. Rabinowitz, Bifurcation, perturbation from simple eigenvalues
		and linearized stability, \emph{Arch. Rat. Mech. Anal.} \textbf{52} (1973), 161--180.
		
	    \bibitem{Da73} E. N. Dancer, Bifurcation Theory for Analytic Operators, \emph{Proc. London Math. Soc.} (3) \textbf{26} (1973), 359--384.
		
		\bibitem{Da732} E. N. Dancer, Global structure of the solutions of non-linear real analytic eigenvalue problems, \emph{Proc. London Math. Soc.} (3) \textbf{27} (1973), 747--765.
		
		\bibitem{Da} E. N. Dancer, Bifurcation from simple eigenvalues and eigenvalues of geometric multiplicity one, \emph{Bull. London Math. Soc.} \textbf{34} (2002), 533--538.
		
		\bibitem{Es} J. Esquinas, Optimal multiplicity in local bifurcation theory, II: General case, \emph{J. Diff. Eqns.} \textbf{75} (1988), 206--215.
		
		\bibitem{ELG} J. Esquinas and J. L\'opez-G\'omez, Optimal multiplicity in local bifurcation theory, I: Generalized generic eigenvalues, \emph{J. Diff. Eqns.} \textbf{71} (1988), 72--92.

\bibitem{FLG} M. Fencl and J. L\'{o}pez-G\'{o}mez, Nodal solutions of weighted indefinite problems,
\emph{J. Evol. Equ.} \textbf{21} (2021), 2815--2835. https://doi.org/10.1007/s00028-020-00625-7
		
		\bibitem{GS} B. Gidas and J. Spruck, A priori bounds for positive solutions of nonlinear elliptic equations, \emph{Comm. in Partial Differential Equations}, \textbf{6(8)} (1981), 883--901.
		
		\bibitem{GS2} B. Gidas and J. Spruck, Global and Local behavior of positive solutions of nonlinear elliptic equations, \emph{Commun. Pure and Appl. Math.}, \textbf{6(8)} (1981).
		
		\bibitem{GT} D. Gilbarg and N. S. Trudinger, \emph{Ellliptic Partial Differential Equations of
			Second Order}, Classics in Mathematics, Springer, Berlin, 1988.
		
	    \bibitem{GGS} I. Gohberg, S. Goldberg and M. A. Kaashoek, \emph{Basic Classes of Linear Operators}, Springer Basel, (2003).
	
	    \bibitem{HP} E. Hille and R. S. Philips, \emph{Functional Analysis and Semi-Groups}, Amer. Math.
	    Soc. Publication XXXI, Amer. Math. Soc., Providence, R. I. (1957).

\bibitem{KQU} U. Kaufmann, H. Ramos-Quoirin, K. Umezu, Loop type subcontinua of positive solutions for indefinite concave-convex problems, \emph{Adv. Nonlinear Stud.} \textbf{19} (2019), 391--412.
		
		\bibitem{Ki} H. Kielh\"ofer, Degenerate Bifurcation at simple Eigenvalues and Stability of Bifurcating Solutions, \emph{J. Functional Anal.} \textbf{38}, (1980) 416--441.
		
		\bibitem{LG01} J. L\'opez-G\'omez, \emph{Spectral Theory and Nonlinear Functional Analysis}, CRC Press, Chapman and Hall RNM vol. 426, Boca Raton, 2001.


		\bibitem{LG13} J. L\'opez-G\'omez, \emph{Linear Second Order Elliptic Operators}, World Scientific, Singapore, 2013.

\bibitem{LGRI} J. L\'{o}pez-G\'{o}mez,  Global bifurcation for Fredholm operators, \emph{Rend. Istit.
Mat. Univ. Trieste} \textbf{48} (2016), 539--564. DOI: 10.13137/2464-8728/13172.

\bibitem{LGEMD} J. L\'{o}pez-G\'{o}mez, J. C. Eilbeck, M. Molina-Meyer, K. Duncan, Structure of solution manifolds in a strongly coupled elliptic system, \emph{IMA J. Numer. Anal.} \textbf{12} (1992), 405--428.
    https://doi.org/10.1093/imanum/12.3.405
		
\bibitem{LGMM} J. L\'{o}pez-G\'{o}mez and M. Molina-Meyer, Bounded components of positive solutions of 		 abstract fixed point equations: mushrooms, loops and isolas, \emph{J. Diff. Equ.} \textbf{209} (2005), 416--441.
		
		\bibitem{LGMC} J. L\'opez-G\'omez and C. Mora-Corral, \emph{Algebraic Multiplicity of Eigenvalues of Linear Operators}, Operator Theory, Advances and Applications vol. 177, Birkh\"auser, Basel, 2007.

\bibitem{JJ3} J. L\'opez-G\'omez and J. C. Sampedro, New analytical and geometrical aspects of the
algebraic multiplicity, \emph{J. Math. Anal. Appns.} \textbf{504} (2021) 125375, pp. 1--20.
		
		\bibitem{JJ4} J. L\'opez-G\'omez and J. C. Sampedro, Bifurcation Theory for Fredholm Operators, \emph{Submitted to JDE}, arXiv:2105.12193v1[math.AP] 25 May 2021.
		
		\bibitem{LGT} J. L\'{o}pez-G\'{o}mez and A. Tellini, Generating an arbitrarily large number of isolas in a superlinear indefinite problem, \emph{Nonl. Anal.} \textbf{108} (2014), 223--248.
		
		\bibitem{MC} C. Mora-Corral, On the Uniqueness of the Algebraic Multiplicity, \emph{J. London Math. Soc.} \textbf{69} (2004), 231--242.
		
		\bibitem{Ra}  P. H. Rabinowitz, Some global results for nonlinear eigenvalue problems, \emph{J.
			Funct. Anal.} \textbf{7} (1971), 487--513..
		
	\end{thebibliography}
\end{document}